\documentclass[runningheads,12pt]{llncs}
\usepackage [english]{babel}
\usepackage[utf8]{inputenc}
\usepackage[T1]{fontenc}

\usepackage{amssymb}
\usepackage{amsmath}
\usepackage{amsfonts}

\usepackage{graphicx}
\usepackage{url}

\usepackage[usenames]{color}

\usepackage{parskip}

\usepackage{xifthen}

  \newtheorem{assumption}{Assumption}

  \usepackage{geometry}
    \newcommand{\keywords}[1]{\par\addvspace\baselineskip  
     \noindent\keywordname\enspace\ignorespaces#1}

\newcommand{\genref}[2]{#2\,\ref{#1}}
\newcommand{\lemref}[1]{\genref{#1}{Lemma}}
\newcommand{\secref}[1]{Section\,\ref{#1}}

\newcommand{\tl}{\|\!|}
\newcommand{\hf}{\p{F}}
\newcommand{\tth}{\tilde{h}}
\newcommand{\hh}{\p{h}}
\newcommand{\qh}{q_h}
\newcommand{\tqh}{\tilde{q}_h}
\DeclareMathOperator{\mdiv}{div}
\DeclareMathOperator{\grad}{\nabla}

\newcommand{\p}[1]{\hat{#1}}
\newcommand{\g}[1]{#1}

\newcommand{\sMP}[1][]{\ifthenelse{\isempty{#1}}{^{(k)}}{^{(#1)}}}

\newcommand{\ifSet}{\iSet_{\mathcal{F}}}
\newcommand{\iSet}{{\mathcal{I}}}

 \newcommand{\MatTwo}[4]
  {\begin{bmatrix}
    #1 & #2\\
    #3 & #4
  \end{bmatrix}}
  
  \newcommand{\VecTwo}[2]
  {\begin{bmatrix}
    #1 \\
    #2 
  \end{bmatrix}}
  
  \newcommand{\ho}[1]{\textcolor{black}{#1}}

\begin{document}

\urldef{\mailsa}\path|christoph.hofer@jku.at|

\title{Analysis of discontinuous Galerkin dual-primal isogeometric tearing and interconnecting methods}
\author{Christoph Hofer$^1$}
\titlerunning{Analysis of dG-IETI-DP}
\authorrunning{C. Hofer}
\institute{ $^1$ Johannes Kepler University (JKU),
				Altenbergerstr. 69, A-4040 Linz, Austria,\\
\mailsa 
 }

\noindent
\maketitle
\begin{abstract}
	In this paper, we present the analysis of the discontinuous Galerkin dual-primal isogeometric tearing and interconnecting method (dG-IETI-DP) 
	for the two-dimensional case where we only consider 
	vertex primal variables.
	The dG-IETI-DP method is a combination of the dual-primal isogeometric tearing and interconnecting method (IETI-DP) with the discontinuous Galerkin (dG) method.  We use the dG method only on the interfaces to couple different patches. This enables us to handle non-matching grids on patch interfaces as well as segmentation crimes (gaps and overlaps) between the patches. The purpose of this paper is to 
	derive
	quasi-optimal  bounds for the condition number of the preconditioned system with respect to the maximal ratio $H/h:=\max_k(H_k/h_k)$ of subdomain diameter and meshsize. We show that the constant  is independent of $h_k$ and $H_k$, but depends on the ratio of meshsizes of neighbouring patches $h_\ell/h_k$.
\end{abstract}

\section{Introduction}
Isogeometric analysis (IgA) is a new methodology for the numerical solution of partial differential equations (PDEs) using the same basis for 
both describing the computational domain and representing the solution. 
IgA was introduced by Hughes, Cottrell and Bazilevs in \cite{HL:HughesCottrellBazilevs:2005a},
and has become a very active field of research,
see also 
\cite{HL:BazilevsVeigaCottrellHughesSangalli:2006a}
for the first results on the numerical analysis of IgA,
the monograph  
\cite{HL:CotrellHughesBazilevs:2009a} 
for a comprehensive presentation of the IgA,
and the recent survey article 
\cite{HL:BeiraodaVeigaBuffaSangalliVazquez:2014a}
on the mathematical analysis of variational isogeometric methods. 
A common choice for basis functions are the so called B-Splines and non-rational uniform B-Splines (NURBS), which are based on a tensor product representation. 
In order to perform local refinements in an efficient way, one has to consider different classes of Splines, e.g., Hierarchical B-Splines (HB-Splines), Truncated HB-Splines (THB-Splines) and T-Splines, see, e.g.,
\cite{HL:VuongGiannelliJuettlerSimeon:2011},
\cite{HL:GiannelliJuettlerSpeleers:2012a}, 
 and \cite{HL:BazilevsCaloCottrellEvans:2010a}, respectively. Moreover, IgA provides a suitable frame for the discretization of a PDE with high-order elements, while having a small number of degrees of freedom.

In the IgA framework, complicated geometries are decomposed into simple domains, called \emph{patches}, which are topologically equivalent to a cube. However, this procedure may introduce small gaps and overlapps at the patch interfaces, leading to so called \emph{segmentation crimes}, see \cite{HL:JuettlerKaplNguyenPanPauley:2014a},
 \cite{HL:PauleyNguyenMayerSpehWeegerJuettler:2015a} and 
 \cite{Hoschek_Lasser_CAD_book_1993} for a more comprehensive analysis. In order to solve PDEs on such domains, numerical schemes based on the discontinuous Galerkin (dG) method for elliptic PDEs 
 were developed and analysed in \cite{HL:HoferLangerToulopoulos:2016a}, \cite{HL:HoferToulopoulos:2016a} and \cite{HL:HoferLangerToulopoulos:2016b}. Moreover, the dG formulation is used when considering different B-Splines spaces across interfaces, e.g., non-matching grids or different spline degrees. An analysis of the dG-IgA formulation with extensions to low regularity solutions can be found in \cite{HL:LangerToulopoulos:2015a}. For a detailed discussion of dG methods, we refer, e.g., to \cite{HL:Riviere:2008a} and \cite{HL:PietroErn:2012a}.
 
 In this paper, we consider fast solution techniques for the system of linear equations arising from the IgA discretization of an elliptic PDE. Our approach is based on the tearing and interconnecting technology, which can be interpreted as a divide and conquer algorithm. We consider the adaptation of the dual-primal finite element tearing and interconnecting (FETI-DP) to the IgA framework, called dual-primal isogeometric tearing and interconnecting (IETI-DP), established in \cite{HL:KleissPechsteinJuettlerTomar:2012a}.  To be more precise, since we use the dG method to couple the different patches we consider its adaption to the dG-IgA formulation, introduced in \cite{HL:HoferLanger:2016a} and denoted by dG-IETI-DP. An application of the dG-IETI-DP method to domains with small gaps and overlaps can be found in \cite{HL:HoferLangerToulopoulos:2016b}. For a comprehensive study and theoretical analysis of FETI-DP and the equivalent Balancing Domain Decomposition by Constraints (BDDC) method, we refer to \cite{HL:ToselliWidlund:2005a}, \cite{HL:Pechstein:2013a} and references therein. The first analysis for the IETI-DP method was done in \cite{HL:VeigaChoPavarinoScacchi:2013a} and extended in \cite{HL:HoferLanger:2016b}. The combination of the FETI-DP method and dG on the interfaces was first introduced and analysed in \cite{HL:DryjaGalvisSarkis:2013a} and \cite{HL:DryjaSarkis:2014a}, see also \cite{HL:DryjaGalvisSarkis:2007a} for an analysis of the corresponding BDDC preconditioner.\\
 Moreover, we refer to other types of efficient solver for IgA systems. We mention overlapping Schwarz methods, see, e.g., \cite{HL:VeigaChoPavarinoScacchi:2012a}, \cite{HL:VeigaChoPavarinoScacchi:2013b},  \cite{HL:BercovierSoloveichik:2015}, and isogeometric mortaring discretizations, see \cite{HL:HeschBetsch:2012a}. In particular, we want to highlight recent advances in multigrid methods for IgA in \cite{HL:HofreitherTakacs:2016a}. There a smoother is constructed based on a stable splitting of the spline space leading to a multigrid method, which is robust with respect to the spline degree in arbitrary dimensions. \\
 The purpose of this paper is to present the analysis for the dG-IETI-DP method. Our proof follows the structure presented in \cite{HL:VeigaChoPavarinoScacchi:2013a} and  \cite{HL:HoferLanger:2016b}.
 We note that, in the analysis, we restrict ourselves to two-dimensional domains having only vertex primal variables,  homogeneous diffusion coefficient and consider only the case of coefficient scaling. Let $\Omega\sMP$ be a patch of the computational domain $\Omega$, $H_k$ be its diameter and $h_k$ characteristic meshsize. We can show that the condition number of the preconditioned system is bounded by $O((1+\log(H/h))^2 q_h^2)$, where $H/h:= \max_k(H_k/h_k)$, $q_h:=\max_{k,\ell}q(h_\ell/h_k)$ with $q(z)=(z+z^2)$ and the hidden constant is independent of $H_k$ and $h_k$. We obtain a quasi-optimal condition number bound with respect to $H/h$ and polynomial bound with respect to the ratio of mesh sizes $h_\ell/h_k$. The quantity $q_h$ in the final theorem   only needs to take into account the ratio of neighbouring meshsizes. The framework used in \cite{HL:VeigaChoPavarinoScacchi:2013a} and  \cite{HL:HoferLanger:2016b} holds for the BDDC preconditioner. Since the BDDC preconditioner and the FETI-DP method have the same spectrum, see \cite{HL:MandelDohrmannTezaur:2005a}, the result applies also to the corresponding IETI-DP method. 
 
 In the present paper, we consider the following 
second-order elliptic boundary value problem 
in a bounded Lipschitz domain $\Omega\subset \mathbb{R}^2$,
as a typical model problem:
Find $u: \overline{\Omega} \rightarrow \mathbb{R}$ such that\\ 
\begin{equation}
  \label{equ:ModelStrong}
  - \mdiv(\alpha \grad u) = f \; \text{in } \Omega,\;
  u = 0 \; \text{on } \Gamma_D,  \;\text{and}\;
  \alpha \frac{\partial u}{\partial n} = g_N \; \text{on } \Gamma_N,
\end{equation}
with given, sufficient smooth data $f, g_N \text{ and } \alpha = \mbox{const} > 0$.
The boundary 
$\partial \Omega$ 
of the computational domain $\Omega$
consists of a Dirichlet part $\Gamma_D$ of positive boundary measure 
and a Neumann part $\Gamma_N$.
Furthermore, we assume that the Dirichlet boundary $\Gamma_D$ is always 
a union of complete domain sides (edges in 2D) 
which are uniquely defined in IgA.
Without loss of generality, we assume 
homogeneous
Dirichlet conditions. 
This can always be  obtained by homogenization.
By means of integration by parts, 
we arrive at
the weak formulation of \eqref{equ:ModelStrong} 
which reads as follows:
Find $u \in V_{D} = \{ u\in H^1: \gamma_0 u = 0 \text{ on } \Gamma_D \}$
such that
\begin{align}
  \label{equ:ModelVar}
    a(u,v) = \left\langle F, v \right\rangle \quad \forall v \in V_{D},
\end{align}
where $\gamma_0$ denotes the trace operator. The bilinear form
$a(\cdot,\cdot): V_{D} \times V_{D} \rightarrow \mathbb{R}$
and the linear form $\left\langle F, \cdot \right\rangle: V_{D} \rightarrow \mathbb{R}$
are given by the expressions
\begin{equation*}
a(u,v) := \int_\Omega \alpha \nabla u \nabla v \,dx
\quad \mbox{and} \quad
\left\langle F, v \right\rangle := \int_\Omega f v \,dx + \int_{\Gamma_N} g_N v \,ds.
\end{equation*}
The remainder of the paper is organized as follows. In \secref{sec:dg_scheme}, we recall the notion introduced in \cite{HL:HoferLanger:2016a} and formulate the dG-IETI-DP method. \secref{sec:Preliminary_results} and \secref{sec:cond_bound} are the main sections of this paper. \secref{sec:Preliminary_results} covers some preliminary theoretical results and introduces the required technical notation. In \secref{sec:cond_bound}, we apply the abstract framework to this problem and obtain the condition number bound for the preconditioned system. Finally, in \secref{sec:conclusion} we draw some conclusions.
\section{Discontinuous Galerkin for Isogeometric Analysis}
In this section we give a very short overview about IgA and dG for IgA. For a more comprehensive study, we refer to, e.g., \cite{HL:CotrellHughesBazilevs:2009a} and \cite{HL:LangerToulopoulos:2015a}.

Let $\p{\Omega}:=(0,1)^d$, where $d\in\{2,3\}$, be the d-dimensional unit cube, which we refer to as the \emph{parameter domain}. Let $p_\iota$ and $M_\iota,\iota\in\{1,\ldots,d\}$, be the B-Spline degree and the number of basis functions along in $x_\iota$-direction. Moreover, let $\Xi_\iota = \{\xi_1=0,\xi_2,\ldots,\xi_{n_\iota}=1\}$, $n_\iota=M_\iota-p_\iota-1$, be a partition of $[0,1]$, called \emph{knot vector}. With this ingredients we are able to define the B-Spline basis $\p{N}_{i,p}$, $i\in\{1,\ldots,M_\iota\}$ on $(0,1)$ via Cox-De Boor's algorithm, cf. \cite{HL:CotrellHughesBazilevs:2009a}. The generalization to $\p{\Omega}$ is performed by a tensor product, again denoted by $\p{N}_{i,p}$, where $i=(i_1,\ldots,i_d)$ and $p=(p_1,\ldots,p_d)$ are a multi-indices. For notational simplicity, we define 
	$\iSet:= \{(i_1,\ldots,i_d)\,|\,i_\iota \in \{1,\ldots,M_\iota\}\} $ as the set of multi-indices.
Since the tensor product knot vector $\Xi$ provides a partition of $\p{\Omega}$, it introduces a mesh $\p{\mathcal{Q}}$, and we denote a mesh element by $\p{Q}$, called \emph{cell}.

The B-Spline functions are used to define our computational domain  $\Omega$, also called \emph{physical domain}. It is given as image of the \emph{geometrical mapping} $G :\; \p{\Omega} \rightarrow \mathbb{R}^{{d}}$, defined as
\begin{align*}
	    G(\xi) := \sum_{i\in \mathcal{I}} P_i \p{N}_{i,p}(\xi),
\end{align*}
with the control points $P_i \in \mathbb{R}^{{d}}$, $i\in \mathcal{I}$. The image of the mesh $\p{\mathcal{Q}}_h$ under $G$ defines the mesh on $\Omega$, denoted by $\mathcal{Q}_h$ with cells $Q$. Both meshes possess a characteristic mesh size $\p{h}$ and $h$, respectively. More complicated geometries $\Omega$ have to be represented with multiple non-overlapping domains $\Omega\sMP:=G\sMP(\p{\Omega}),k=1,\ldots,N$, called \emph{patches}, where each patch is associated with a different geometrical mapping $G\sMP$. We sometimes call $\overline{\Omega}:=\bigcup_{k=1}^N\overline{\Omega}\sMP$ a \emph{multipatch domain}. Furthermore, we denote the set of all indices $\ell$  
such that $\Omega^{(k)}$ and $\Omega^{(\ell)}$ have a common interface $F^{(k\ell)}$ by ${\mathcal{I}}_{\mathcal{F}}^{(k)}$.
We define the interface $\Gamma\sMP$ of $\Omega\sMP$ as $\Gamma\sMP := \bigcup_{\ell\in\ifSet\sMP}^N F^{(k\ell)}$.

The B-Splines are also used for approximating the solution of our PDE. This motivates to define the basis functions in the physical space as $\g{N}_{i,p}:=\p{N}_{i,p}~\circ~G^{-1}$ and the corresponding discrete space as
\begin{align}
\label{equ:gVh}
  V_h:=\text{span}\{\g{N}_{i,p}\}_{i\in\iSet}.
\end{align}
Moreover, each function $u(x) = \sum_{i\in\mathcal{I}} u_i \g{N}_{i,p}(x)$ 
is associated with the 
 coefficient  vector $\boldsymbol{u} = (u_i)_{i\in\mathcal{I}}$. 
This map is known as \emph{Ritz isomorphism} 
or \emph{IgA isomorphism} in connection with IgA, 
One usually writes this relation as 
$u_h \leftrightarrow \boldsymbol{u}$,
 and we will use it in the following without further comments. If we consider a single patch $\Omega^{(k)}$ of a multipatch domain $\Omega$, we will use the notation $V_{h}^{(k)},\g{N}_{i,p}^{(k)},\p{N}_{i,p}^{(k)}, G^{(k)}, \ldots$ with the analogous definitions. To keep notation simple, we will use $h_k$ and $  \p{h}_k$ instead of $h\sMP$ and $\p{h}\sMP$, respectively.

In this paper we consider the dG-IgA scheme, where we use the 
spaces $V_{h}^{(k)}$ of continuous functions on each patch $\Omega^{(k)}$, 
whereas 
discontinuities are allowed 
across the patch interfaces.
The continuity of the function values and its normal fluxes are enforced in a weak sense by adding additional terms to the bilinear form.
For the remainder of this paper, we define
the dG-IgA space
\begin{align}
\label{equ:gVh_glob}
  V_{h}:= V_{h}(\Omega) := \{v\,| \,v|_{\Omega^{(k)}}\in V_{h}^{(k)}\},
\end{align}
where $V_{h}^{(k)}$ is defined as in \eqref{equ:gVh}. 
A comprehensive study of dG schemes for FE can be found in \cite{HL:Riviere:2008a} and \cite{HL:PietroErn:2012a}. For an analysis of the dG-IgA scheme, we refer to \cite{HL:LangerToulopoulos:2015a}.

For simplicity of the presentation, we assume that we have homogeneous Dirichlet boundary condition. Hence, we define $V_{D,h}$ as the space of all functions from $V_{h}$ which vanish on the Dirichlet boundary $\Gamma_D$. Having these definitions at hand, we can define the discrete problem based on the Symmetric Interior Penalty (SIP) dG formulation as follows:
Find $u_h \in V_{D,h}$ such that
  \begin{align}
  \label{equ:ModelDiscDG}
    a_h(u_h,v_h) = \left\langle F, v_h \right\rangle \quad \forall v_h \in V_{D,h},
  \end{align}
where
\begin{align*}
    a_h(u,v) &:= \sum_{k=1}^N a_e^{(k)}(u,v) \quad \text{and} \quad \left\langle F, v \right\rangle:=\sum_{k=1}^N \left(\int_{\Omega^{(k)}}f v^{(k)} dx+\int_{\Gamma_N\sMP} g_N v\sMP \,ds\right),\\
    a^{(k)}_e(u,v) &:=  a^{(k)}(u,v) + s^{(k)}(u,v) + p^{(k)}(u,v),
\end{align*}
and
\begin{align*}
a^{(k)}(u,v) &:= \int_{\Omega^{(k)}}\alpha^{(k)} \nabla u^{(k)} \nabla v ^{(k)} dx,\\
    s^{(k)}(u,v)&:= \sum_{\ell\in{\mathcal{I}}_{\mathcal{F}}^{(k)}} \int_{F^{(k\ell)}}\frac{\alpha^{(k)}}{2}\left(\frac{\partial u^{(k)}}{\partial n}(v^{(\ell)}-v^{(k)})+ \frac{\partial v^{(k)}}{\partial n}(u^{(\ell)}-u^{(k)})\right)ds,\\ 
    p^{(k)}(u,v)&:= \sum_{\ell\in{\mathcal{I}}_{\mathcal{F}}^{(k)}} \int_{F^{(k\ell)}}\frac{\delta \alpha^{(k)}}{h_{k\ell}}(u^{(\ell)}-u^{(k)})(v^{(\ell)}-v^{(k)})\,ds.
\end{align*}
The notation $\frac{\partial}{\partial n}$ means the derivative 
in the direction of the outer normal vector,
$\delta$ is a positive sufficiently large penalty parameter,  and $h_{k\ell}$ is the harmonic average of the adjacent mesh sizes, i.e., $h_{k\ell}= 2h_k h_\ell/(h_k + h_\ell)$.

We equip $V_{D,h}$ with the dG-norm 
\begin{align}
\label{HL:dgNorm}
 \left\|u\right\|_{dG}^2  = \sum_{k = 1}^N\left[\alpha^{(k)} \left\|\nabla u^{(k)}\right\|_{L^2(\Omega^{(k)})}^2 + \sum_{\ell\in{\mathcal{I}}_{\mathcal{F}}^{(k)}} \frac{\delta \alpha^{(k)}}{h_{k\ell}}\int_{F^{(k\ell)}} (u^{(k)} - u^{(\ell)})^2 ds\right].
\end{align}

Furthermore,
we define the bilinear forms
\begin{align*}
 d_h(u,v) = \sum_{k =1}^N d^{(k)}(u,v) \quad \text{where} \quad d^{(k)}(u,v)= a_e^{(k)}(u,v) + p^{(k)}(u,v),
\end{align*}
for later use.
We note that $\left\|u_h\right\|_{dG}^2 = d_h(u_h,u_h)$.
%
\begin{lemma}
\label{lem:wellPosedDg}
Let $\delta$ be sufficiently large. 
Then there exist two positive constants $\gamma_0$ and $\gamma_1$ 
which are independent of $h_k,H_k,\delta,\alpha^{(k)}$ and $u_h$
such that the inequalities
 \begin{align}
 \label{equ:equivPatchNormdG}
  \gamma_0 d^{(k)}(u_h,u_h)\leq a_e^{(k)}(u_h,u_h) \leq \gamma_1 d^{(k)}(u_h,u_h), \quad 
  \forall u_h\in V_{D,h}
 \end{align}
are valid for all $k=1,2,\ldots,N$. Furthermore, we have  the inequalities
\begin{align}
 \label{equ:equivNormdG}
 \gamma_0 \left\|u_h\right\|_{dG}^2\leq a_h(u_h,u_h)\leq \gamma_1 \left\|u_h\right\|_{dG}^2, \quad 
  \forall u_h\in V_{D,h}.
\end{align}

\end{lemma}
This Lemma is an equivalent statement of Lemma 2.1 in \cite{HL:DryjaGalvisSarkis:2013a} for IgA,
and the proof can be found in \cite{HL:HoferLanger:2016a}. A direct implication of \eqref{equ:equivNormdG} is the well posedness of the discrete problem \eqref{equ:ModelDiscDG} by the Theorem of Lax-Milgram. The consistency of the method together with interpolation estimates for B-splines lead to an a-priori error estimate, established in \cite{HL:LangerToulopoulos:2015a}.
%
%
We note that, in \cite{HL:LangerToulopoulos:2015a}, the results were obtained 
for the Incomplete Interior Penalty (IIP) scheme. 
An extension to SIP-dG and the use of  harmonic averages for $h$ and/or $\alpha$ 
are discussed in Remark~3.1 in \cite{HL:LangerToulopoulos:2015a},
see also \cite{HL:LangerMantzaflarisMooreToulopoulos:2015b}.

We choose the B-Spline functions $\{\g{N}_{i,p}\}_{i\in\mathcal{I}_0}$ as basis for the space $V_{h}$, see \eqref{equ:gVh_glob}, where $\mathcal{I}_0$ contains all indices of $\mathcal{I}$, where the corresponding B-Spline basis functions do not have a support on the Dirichlet boundary. Hence, the dG-IgA scheme $\eqref{equ:ModelDiscDG}$ is equivalent to the system of linear equations
\begin {align}
\label{equ:Ku=f_DG}
  \boldsymbol{K} \boldsymbol{u} = \boldsymbol{f},
\end{align}
where 
$\boldsymbol{K} = (\boldsymbol{K}_{i,j})_{i,j\in {\mathcal{I}}_0}$
and
$\boldsymbol{f}= (\boldsymbol{f}_i)_{i\in {\mathcal{I}}_0}$
denote the stiffness matrix and the load vector, respectively,
with 
$ \boldsymbol{K}_{i,j} = a(\g{N}_{j,p},\g{N}_{i,p})$
and 
$\boldsymbol{f}_i = \left\langle F, \g{N}_{i,p} \right\rangle$, and
$\boldsymbol{u}$ is the vector representation of $u_h$. 
\section{IETI-DP for dG-IgA}
\label{sec:dg_scheme}
In this section we rephrase the main ingredients for the dG-IETI-DP method in two dimensions and provide definitions used in the analysis in \secref{sec:Preliminary_results} and \secref{sec:cond_bound}. A more sophisticated presentation of the method can be found in \cite{HL:HoferLanger:2016a}. 
\subsection{Basic setup and local space description}
\label{sec:LocSpaces}
 In order to keep the presentation  simple, we assume that the considered patch $\Omega^{(k)}$ does not touch the Dirichlet boundary. The other case can be handled in an analogous way.  Note, although $ \overline{F}_{\ell k} \subset \partial\Omega^{(\ell)}$ and  $\overline{F}_{k l}\subset \partial\Omega^{(k)}$ are geometrically the same, they are treated as different objects.
 
 
For each patch $\Omega^{(k)}$, we define its extended version $\Omega^{(k)}_e$ via the union with all neighbouring interfaces $\overline{F}_{\ell k}\subset\partial\Omega^{(\ell)}$  and similarly we introduce also the extended interface $\Gamma^{(k)}_e$:
\begin{align*}
	\overline{\Omega}^{(k)}_e := \overline{\Omega}^{(k)} \cup \{\bigcup_{\ell\in{\mathcal{I}}_{\mathcal{F}}^{(k)}} \overline{F}^{(\ell k)}\}, \quad \Gamma^{(k)}_e := \Gamma^{(k)} \cup \{\bigcup_{\ell\in{\mathcal{I}}_{\mathcal{F}}^{(k)}} \overline{F}^{(\ell k)}\}.
\end{align*}
%
%
%
Moreover, based on the definitions above, we introduce the following quantities for the whole multipatch domain
\begin{align*}
\overline{\Omega}_e := \bigcup_{k=1}^{N} \overline{\Omega}_e^{(k)}, \quad  \Gamma := \bigcup_{k = 1}^N \Gamma^{(k)} \text{ and}\quad \Gamma_e := \bigcup_{k = 1}^N\Gamma^{(k)}_e.
\end{align*}
The next step is to describe appropriate discrete function spaces to reformulate \eqref{equ:ModelDiscDG} in order to treat the new formulation in the spirit of the (classical) IETI-DP method.  We start with a description of the discrete function spaces for a single patch. 

As defined in $\eqref{equ:gVh}$, let $V_{h}^{(k)}$ be the discrete function space 
defined on the patch $\Omega^{(k)}$. 
Then we define the corresponding function space for the extended patch $\Omega^{(k)}_e$ by
\begin{align*}
  V_{h,e}^{(k)} := V_{h}^{(k)} \times \prod_{\ell\in{\mathcal{I}}_{\mathcal{F}}^{(k)}}V_{h}^{(k)}(\overline{F}^{(\ell k)}),
\end{align*}
where $V_{h}^{(k)}(\overline{F}^{(\ell k)}) \subset V_{h}^{(\ell)}$ is given by
\begin{align*}
 V_{h}^{(k)}(\overline{F}^{(\ell k)}) := \text{span}\{\g{N}_{i,p}^{(\ell)} \,|\, \text{supp}\{\g{N}_{i,p}^{(\ell)}\}\cap\overline{F}^{(\ell k)} \neq \emptyset\}. 
 \end{align*}
According to the notation introduced in \cite{HL:DryjaGalvisSarkis:2013a}, we will represent a function $u^{(k)} \in V_{h,e}^{(k)} $ as
\begin{align}
 u^{(k)} = \{u\sMP[k,k], \{u\sMP[k,\ell]\}_{\ell\in {\mathcal{I}}_{\mathcal{F}}^{(k)}}\},
\end{align}
where $u\sMP[k,k]$ and $u\sMP[k,\ell]$  are the restrictions of $u^{(k)}$ to $\Omega^{(k)}$ and $\overline{F}^{(\ell k)}$, respectively. By introducing a suitable ordering, a function $u\sMP\in V_{h,e}$ possesses a vector representation $\boldsymbol{u}\sMP = (u_i\sMP)_{i\in\iSet\sMP_e}$. Moreover, we introduce an additional representation of $u^{(k)}\in V_{h,e}^{(k)} $, as $u^{(k)} = (u^{(k)}_I, u^{(k)}_{B_e})$, where 
\begin{align*}
 u^{(k)}_I\in V_{I,h}^{(k)}:=V_{h}^{(k)} \cap H^1_0(\Omega^{(k)}),
\end{align*}
and
\begin{align*}
  u^{(k)}_{B_e} \in W^{(k)}:=\text{span}\{\g{N}_{i,p}^{(\ell)}\,|\, \,\text{supp}\{\g{N}_{i,p}^{(\ell)}\}\cap\Gamma^{(k)}_e \neq \emptyset \text{ for } \ell\in{\mathcal{I}}_{\mathcal{F}}^{(k)}\cup\{k\}\}.
\end{align*}
This provides a representation of $V_{h,e}^{(k)} $ 
in the form of
$V_{I,h}^{(k)} \times W^{(k)}$. 

\subsection{ Schur complement and discrete harmonic extensions}
\label{sec:DHE_Schur}

We note that the patch local bilinear form $a^{(k)}_e(\cdot,\cdot)$ is defined on the space $V_{h,e}^{(k)}\times V_{h,e}^{(k)}$, since it requires function values of the neighbouring patches $\Omega^{(\ell)},\ell\in{\mathcal{I}}_{\mathcal{F}}^{(k)}$. Therefore, it depicts a matrix representation $\boldsymbol{K}_e^{(k)}$ 
satisfying the identity
\begin{align*}
 a^{(k)}_e(u^{(k)},v^{(k)}) = (\boldsymbol{K}_e^{(k)} \boldsymbol{u},\boldsymbol{v})_{l_2}\quad \text{for } u^{(k)},v^{(k)} \in V_{h,e}^{(k)},
\end{align*}
where $\boldsymbol{u}$ and $\boldsymbol{v}$ denote the  vector representation of $u^{(k)}$ and $v^{(k)}$, respectively. By means of the representation $V_{I,h}^{(k)} \times W^{(k)}$ for $V_{h,e}^{(k)} $, we can structure the matrix $\boldsymbol{K}^{(k)}_e$ in the following way
\begin{align}
\label{equ:loc_K}
 \boldsymbol{K}^{(k)}_e=
 \begin{bmatrix}
  \boldsymbol{K}^{(k)}_{e,II} & \boldsymbol{K}^{(k)}_{e,IB_e} \\
  \boldsymbol{K}^{(k)}_{e,B_e I} & \boldsymbol{K}^{(k)}_{e,B_e B_e }
 \end{bmatrix}.
\end{align}
This 
enables 
us to define the Schur complement of $\boldsymbol{K}^{(k)}_e$ with respect to $W^{(k)}$ as
\begin{align}
\label{equ:loc_Schur}
 \boldsymbol{S}^{(k)}_e := \boldsymbol{K}^{(k)}_{e,B_e B_e } - \boldsymbol{K}^{(k)}_{e,B_e I}\left(\boldsymbol{K}^{(k)}_{e,II}\right)^{-1}\boldsymbol{K}^{(k)}_{e,IB_e}.
\end{align}
We denote 
the corresponding bilinear form 
by $s^{(k)}_e(\cdot,\cdot)$,
and 
the corresponding operator
by $S^{(k)}_e: W^{(k)}\to{W^{(k)}}^*$, i.e.
\begin{align*}
 (\boldsymbol{S}^{(k)}_e \boldsymbol{u}_{B_e}^{(k)},\boldsymbol{v}_{B_e}^{(k)})_{l^2} = \langle S^{(k)}_e u_{B_e}^{(k)},v_{B_e}^{(k)} \rangle = s^{(k)}_e(u_{B_e}^{(k)},v_{B_e}^{(k)}), \quad \forall u_{B_e}^{(k)},v_{B_e}^{(k)}\in W^{(k)}.
\end{align*}
The Schur complement has the property that
\begin{align}
\label{equ:SchurMin}
 \langle S^{(k)}_e u_{B_e}^{(k)},u_{B_e}^{(k)}\rangle = \min_{w^{(k)}= (w_I^{(k)},w_{B_e}^{(k)})\in V_{h,e}^{(k)}} a^{(k)}_e(w^{(k)},w^{(k)}),
\end{align}
such that $w_{B_e}^{(k)} = u_{B_e}^{(k)}$ on $\Gamma^{(k)}_e$. 
We define the \emph{discrete NURBS harmonic extension} $\mathcal{H}^{(k)}_e$ (in the sense of $a^{(k)}_e(\cdot,\cdot)$) for patch $\Omega^{(k)}_e$ by
\begin{align}
\label{def:DHEext_loc}
\begin{split}
 \mathcal{H}^{(k)}_e&: W^{(k)} \to V_{h,e}^{(k)}:\\
&\begin{cases}
 \text{Find }\mathcal{H}^{(k)}_e{u_{B_e}}\in V_{h,e}^{(k)}: & \\
  \quad a^{(k)}_e(\mathcal{H}^{(k)}_e{u_{B_e}},u^{(k)})=0 \quad &\forall u^{(k)}\in V_{I,h}^{(k)},\\
   \quad \mathcal{H}^{(k)}_e{u_{B_e}}_{|\Gamma_e^{(k)}} = {u_{B_e}}_{|\Gamma_e^{(k)}}, &
\end{cases}
\end{split}
\end{align}
where 
$V_{I,h}^{(k)}$ 
is here interpreted as subspace of $V_{h,e}^{(k)}$ with vanishing function values on $\Gamma_e^{(k)}$. One can show that the minimizer in \eqref{equ:SchurMin} is given by $\mathcal{H}^{(k)}_e{u_{B_e}}$. 
In addition, we introduce the \emph{standard discrete NURBS harmonic extension} $\mathcal{H}^{(k)}$ (in the sense of $a^{(k)}(\cdot,\cdot)$) of $u^{(k)}_{B_e}$ as follows:
\begin{align}
\label{def:DHE_loc}
\begin{split}
 \mathcal{H}^{(k)}&: W^{(k)} \to V_{h,e}^{(k)}:\\
&\begin{cases}
 \text{Find }\mathcal{H}^{(k)}{u_{B_e}}\in V_{h,e}^{(k)}: & \\
  \quad a^{(k)}(\mathcal{H}^{(k)}{u_{B_e}},u^{(k)})=0 \quad &\forall u^{(k)}\in V_{I,h}^{(k)},\\
   \quad \mathcal{H}^{(k)}{u_{B_e}}_{|\Gamma_e^{(k)}} = {u_{B_e}}_{|\Gamma_e^{(k)}}, &
\end{cases}
\end{split}
\end{align}
where $V_{I,h}^{(k)}$ is the same space as in \eqref{def:DHEext_loc}, 
and $a^{(k)}(\cdot,\cdot)$ 
is interpreted as a
bilinear form on the space 
$V_{h,e}^{(k)}\times V_{h,e}^{(k)}$. 
The crucial point
is to show equivalence in the energy norm $d_h(u_h,u_h)$ between functions, 
which are discrete harmonic in the sense of $\mathcal{H}^{(k)}_e$ and $\mathcal{H}^{(k)}$. 
This property is summarized in the following Lemma, 
cf. also Lemma~3.1 in \cite{HL:DryjaGalvisSarkis:2013a}.
%
%
\begin{lemma}
\label{lem:equDiscHarmonic}
There exists a positive constant 
which  is independent of $\delta, h_k, H_k, \alpha^{(k)}$ and $u_{B_e}^{(k)}$
such that the inequalities
 \begin{align}
  d^{(k)}(\mathcal{H}^{(k)}{u_{B_e}},\mathcal{H}^{(k)}{u_{B_e}})\leq d^{(k)}(\mathcal{H}_e^{(k)}{u_{B_e}},\mathcal{H}_e^{(k)}{u_{B_e}})\leq C d^{(k)}(\mathcal{H}^{(k)}{u_{B_e}},\mathcal{H}^{(k)}{u_{B_e}}),
 \end{align}
hold for all $u_{B_e}^{(k)}\in W^{(k)}$.
\end{lemma}
%
%
The subsequent statement immediately follows from \lemref{lem:wellPosedDg} and \lemref{lem:equDiscHarmonic}, see also \cite{HL:DryjaGalvisSarkis:2013a}.
%
%
\begin{corollary}
\label{cor:equivalence_dGNorm_DHE}
The spectral equivalence inequalities
  \begin{align}
  C_0 d^{(k)}(\mathcal{H}^{(k)}{u_{B_e}},\mathcal{H}^{(k)}{u_{B_e}})\leq a_e^{(k)}(\mathcal{H}_e^{(k)}{u_{B_e}},\mathcal{H}_e^{(k)}{u_{B_e}})\leq C_1 d^{(k)}(\mathcal{H}^{(k)}{u_{B_e}},\mathcal{H}^{(k)}{u_{B_e}}),
 \end{align}
hold for all $u_{B_e}^{(k)}\in W^{(k)}$,
where the constants $C_0$ and $C_1$ are independent of 
$\delta, h_k, H_k, \alpha^{(k)}$ and $u_{B_e}^{(k)}$.
\end{corollary}

\subsection{Global space description}
Based on the definitions of the local spaces in \secref{sec:LocSpaces}, we can introduce the space 
\begin{align*}
 V_{h,e} := \{ v \,|\, v^{(k)}\in V_{h,e}^{(k)}, k\in\{1,\ldots,N\}\}
\end{align*}
for the whole extended domain $\Omega_e$.
Additionally, we need a description of the global extended interface spaces
\begin{align*}
 W := \{ v_{B_e} \,|\, v_{B_e}^{(k)}\in W^{(k)}, k\in\{1,\ldots,N\}\} = \prod_{k=1}^N W^{(k)}.
\end{align*}
We note that, according to \cite{HL:DryjaGalvisSarkis:2013a}, we will also interpret this space  as subspace of $V_{h,e}$, where its functions are discrete harmonic in the sense of $\mathcal{H}^{(k)}_e$ on each $\Omega^{(k)}$. For completeness, we define the discrete NURBS harmonic extension in the sense of $\sum_{k=1}^N a^{(k)}_e(\cdot,\cdot)$ and $\sum_{k=1}^N a^{(k)}(\cdot,\cdot)$ for $W$ as $\mathcal{H}_e u =\{\mathcal{H}_e^{(k)} u^{(k)}\}_{k=1}^N$ and $\mathcal{H}_e u =\{\mathcal{H}^{(k)} u^{(k)}\}_{k=1}^N$, respectively.

We aim at reformulating \eqref{equ:ModelDiscDG} and \eqref{equ:Ku=f_DG} in terms of the extended domain $\Omega_e$ and introducing Lagrange multipliers in order to couple of the independent interface dofs. In the context of tearing and interconnecting methods, we need a ``continuous'' subspace $\widehat{W}$ of $W$ such that $\widehat{W}$ is equivalent to $V_{\Gamma,h}$, i.e., $\widehat{W} \equiv V_{\Gamma,h}$. Since the space $V_{\Gamma,h}$ consists of functions which are discontinuous across the patch interface, the common understanding of continuity makes no sense. For an appropriate definition of continuity in the context of the spaces $\widehat{W}, W, V_{\Gamma,h}, V_{h,e}$ and $V_{h}$, we refer to \cite{HL:HoferLanger:2016a}. Similarly, we can define  a ``continuous'' subspace $\widehat{V}_{h,e}\subset V_{h,e}$, such that $\widehat{V}_{h,e}\equiv V_h$.

We can reformulate (\ref{equ:Ku=f_DG}) in the space $\widehat{V}_{h,e}$ yielding the equation $\widehat{\boldsymbol{K}}_e \boldsymbol{u}_e = \widehat{\boldsymbol{f}}_e$. By means of the local Schur complements defined in \secref{sec:DHE_Schur}, we can reformulate this equation as $\widehat{\boldsymbol{S}}_e \boldsymbol{u}_{B_e} = \widehat{\boldsymbol{g}}_e$, where $u_{B_e}\in\widehat{W}$. This equation is equivalent to the following minimization problem
\begin{align}
      \label{equ:min_Schur}
    u_{B_e,h} = \underset{w\in W, Bw = 0}{\text{argmin}} \; \frac{1}{2} \langle S_e w , w\rangle -  \langle g_e,w\rangle,
   \end{align}
   where the operator $B$ enforces the ``continuity'' of $w\in W$, i.e. $\widehat{W}=\ker{B}$, and $\langle S_e w,v\rangle:=\sum_{k=1}^N\langle S_e\sMP w\sMP,v\sMP\rangle$ is the operator representation of $\mathbf{S}_e:=\text{diag}(\mathbf{S}_e\sMP)$.
In the following, we will only work with the Schur complement system.
In order to simplify the notation, we will use $u$ instead of $u_{B_e,h}$, when we consider functions in $V_{\Gamma,h}$. 
If 
we have to made
a distinction between $u_h, u_{B_e,h}$ and $u_{I,h}$, 
we will add the subscripts again.

For the dual-primal variants of the tearing and interconnecting methods, we need a space $\widetilde{W}$ with $\widehat{W} \subset \widetilde{W} \subset W$ and where $S_e$ restricted to $\widetilde{W}$ is positive definite. Let $\Psi \subset V_{\Gamma,h}^*$ be a set of linearly independent \emph{primal variables}.
Then we define the spaces
\begin{equation*}
\widetilde{W} := \{w\in W:  \psi(w^{(k)}) = \psi(w^{(\ell)}), \forall\psi \in \Psi, \forall k>l  \}
\end{equation*}
and
\begin{equation*}
W_{\Delta} := \prod_{k=1}^N W_{\Delta}^{(k)},\text{ with} \quad W_{\Delta}^{(k)}:=\{w^{(k)}\in W^{(k)}:\, 
\psi(w^{(k)}) =0\; \forall\psi \in \Psi\}.
\end{equation*}
Moreover, we introduce the space $W_{\Pi} \subset \widehat{W}$ such that
$\widetilde{W} = W_{\Pi} \oplus W_{\Delta}.$
We call $W_{\Pi}$ \emph{primal space} and $W_{\Delta}$ \emph{dual space}. 
If we choose $\Psi$ such that $\widetilde{W} \cap \ker{S_e}=\{0\}$, then
  \begin{align*}
   \widetilde{S}_e: \widetilde{W} \to \widetilde{W}^*, \, \text{ with } \langle \widetilde{S}_e v,w\rangle =  \langle S_e v,w\rangle \quad \forall v,w \in \widetilde{W},
  \end{align*}
  is invertible. 
Typical choices are continuous vertex values and/or continuous interface averages. A formal definition of the primal variables for dG-IETI-DP method can be found in \cite{HL:HoferLanger:2016a}. In the following analysis, we will restrict ourselves to the case of continuous vertex values, i.e., $\psi^{\mathcal{V}}(v) = v(\mathcal{V})$, where $\mathcal{V}$ is a corner of $\Omega_e\sMP$.

\subsection{IETI - DP and preconditioning}
\label{sec:IETI_DP}
We are now in the position to reformulate the problem \eqref{equ:min_Schur} in $\widetilde{W}$ and write it as saddle point problem  as follows:
Find $(u,\boldsymbol{\lambda}) \in \widetilde{W} \times U:$
    \begin{align}
    \label{equ:saddlePointReg}
     \MatTwo{\widetilde{S}_e}{\widetilde{B}^T}{\widetilde{B}}{0} \VecTwo{u}{\boldsymbol{\lambda}} = \VecTwo{\widetilde{g}}{0},
    \end{align}
    where $\widetilde{S}_e, \widetilde{B}$ and $ \widetilde{g}$ are the corresponding representations in $\widetilde{W}$ and $\widetilde{W}^*$.
%
%
By construction, $\widetilde{S}_e$ is SPD on $\widetilde{W}$. Therefore, we can define the Schur complement $F$ 
and the corresponding right-hand side of equation \eqref{equ:saddlePointReg} 
as follows:
\begin{align*}
    F:= \widetilde{B} \widetilde{S}_e^{-1}\widetilde{B}^T, \quad d:= \widetilde{B}\widetilde{S}_e^{-1} \widetilde{g}.
\end{align*}
Hence, the saddle point system \eqref{equ:saddlePointReg} is equivalent to 
the Schur complement problem:
\begin{align}
   \label{equ:SchurFinal}
      \text{Find } \boldsymbol{\lambda} \in U: \quad F\boldsymbol{\lambda} = d.
\end{align}
Equation \eqref{equ:SchurFinal} is solved by means of the PCG algorithm, but it requires an appropriate preconditioner in order to obtain an efficient solver. According to \cite{HL:DryjaGalvisSarkis:2013a} and \cite{HL:DryjaSarkis:2014a}, the right choice for FE is the \emph{scaled Dirichlet preconditioner} $M_{sD}^{-1}$, 
adapted to
the extended set of dofs. 
In \secref{sec:cond_bound} we will prove, that the scaled Dirichlet preconditioner works well  for the IgA setting too. A formal definition of $M_{sD}^{-1}$ and  numerical experiments confirming this can be found in \cite{HL:HoferLanger:2016a}. 
Since we can consider the dG-IETI-DP method as a conforming Galerkin (cG) method on an extended grid,
 we can implement the dG-IETI-DP algorithm following the implementation of the corresponding 
 cG-IETI-DP method given in \cite{HL:HoferLanger:2016b}.

  In \cite{HL:DryjaGalvisSarkis:2013a} and \cite{HL:DryjaSarkis:2014a}, it is proven for FE that the condition number behaves 
like the condition number of the preconditioned system
for the continuous FETI-DP method, see also \cite{HL:DryjaGalvisSarkis:2007a} for dG-BDDC FE preconditioners. From \cite{HL:HoferLanger:2016b} and \cite{HL:VeigaChoPavarinoScacchi:2013a}, we know that the condition number of the continuous IETI-DP and BDDC-IgA operators 
is also quasi-optimal with respect to the ratio of patch and mesh size. In the next section, we prove that the condition number for the dG-IETI-DP operator behaves as
\begin{align*}
       \kappa(M_ {sD}^{-1} F_{|\text{ker}(\widetilde{B}^T)}) \leq C \max_k\left(1+\log\left(\frac{H_k}{h_k}\right)\right)^2,
\end{align*}
where $H_k$ and $h_k$ are the patch size and mesh size, respectively, 
and the positive  constant $C$ is independent of 
$H_k$, $h_k$, but depends on $h_k/h_\ell$. We use the fact that the IETI-DP method and the BDDC preconditioner have the same spectrum, up to some zeros and ones, which was proven in  \cite{HL:MandelDohrmannTezaur:2005a} based on algebraic arguments. So we will prove the condition number bound for the corresponding BDDC method and the result then also applies to the dG-IETI-DP method. We can use the framework developed in \cite{HL:VeigaChinosiLovadinaPavarino:2010a} also for the dG variant, since the dG-IETI-DP method can be seen as a IETI-DP method on an extended domain $\Omega_e$.  In the next section we provide some auxiliary results, which will be needed for the proof in \secref{sec:cond_bound}.

\section{Preliminary results}
\label{sec:Preliminary_results}
In this section we want to define a discrete norm $|\cdot|_{dG}$ for the space $V_{h,e}^{(k)}$, based on the coefficient vector $\boldsymbol{u} = (u_i)_{i\in\iSet\sMP_e}$, which can be seen as the discrete analogue of the norm induced by $d\sMP(\cdot,\cdot)$. For notational simplicity  we denote this induced norm again by $\|\cdot\|_{dG}$. The difficulty is that the grids on $F\sMP[k\ell]$ and $F\sMP[\ell k]$ do not match and, hence, the coefficients corresponding to that part cannot be directly related. We will resolve that issue using a $L^2$-projection onto $F\sMP[\ell k]$. Although some results are stated with arbitrary dimension $d$, we will always focus on the case $d=2$ with continuous vertex values only. In order to have a clear distinction between the function $u$ and its coefficients $(u_i)_{i\in\iSet_e}$, we denote in the following the coefficients with $(c_i)_{i\in\iSet_e}$.

We rephrase important definitions and results from \cite{HL:VeigaChoPavarinoScacchi:2013a} and \cite{HL:HoferLanger:2016b} with small adjustment due considering the dG-formulation. Let $\g{u}\sMP = \{\g{u}\sMP[k,k], \{\g{u}\sMP[k,\ell]\}_{\ell\in \ifSet\sMP}\}$ be a function in $V_{h,e}\sMP$. The functions $\g{u}\sMP[k,k]$ and $\g{u}\sMP[k,\ell]$ possess a representation of the form
\begin{align}
\label{equ:coeff_dgSpline}
 \g{u}\sMP[k,k] = \sum_{i\in\iSet\sMP[k,k]}c_{i}\sMP[k,k] \g{N}_{i,p}\sMP \text{\quad and\quad } \g{u}\sMP[k,\ell] = \sum_{i\in\iSet\sMP[k,\ell]}c_{i}\sMP[k,\ell]\g{N}_{i,p}\sMP[\ell],
\end{align}
where $\g{u}\sMP[k,k]\in V_{h}^{(k)}$ and $ \g{u}\sMP[k,\ell] \in V_{h,e}^{(k)}(F^{(\ell k)})$. Here, $\iSet\sMP[k,k]$
denotes all indices, such $\g{N}_{i,p}\sMP$ has a support on $\Omega\sMP$ and $\iSet\sMP[k,\ell]$
denotes all indices, such $\g{N}_{i,p}\sMP[\ell]$ has a support on $F\sMP[\ell k]$. Moreover, we define the trace space of $V_{h}\sMP$ on $F\sMP[k\ell]$ as $V_{h}\sMP(F\sMP[k\ell])$, i.e., $\g{u} \sMP[k,k]_{|F\sMP[k\ell]}\in V_{h}\sMP(F\sMP[k\ell])$. We define the parameter domain representation of $\g{u}$ as $\p{u}\sMP = \{\p{u}\sMP[k,k], \{\p{u}\sMP[k,\ell]\}_{\ell\in \ifSet\sMP}\}$, where 
\begin{align}
\label{equ:coeff_dgSplineParameter}
 \p{u}\sMP[k,k] = \sum_{i\in\iSet\sMP[k,k]}c_{i}\sMP[k,k] \p{N}_{i,p}\sMP \text{\quad and\quad } \p{u}\sMP[k,\ell] = \sum_{i\in\iSet\sMP[k,\ell]}c_{i}\sMP[k,\ell]\p{N}_{i,p}\sMP[\ell]
\end{align}
and $c_{i}\sMP[k,k]$ and $c_{i}\sMP[k,\ell]$ are as in \eqref{equ:coeff_dgSpline}.

 Let $h_k$ and $\hh_k$ be the characteristic meshsize in $\Omega\sMP$ and $\p{\Omega}\sMP$, respectively.  
 Since the geometry mapping $G\sMP$ is fixed on a coarse discretization, it is independent of $h_k$. Moreover, by basic properties of $G\sMP$ we can assume, that there exists a constant, independent of $H_k$ and $h_k$, such that
\begin{align}
\label{equ:hath_h}
  C^{-1} \hh_k \leq h_k/H_k \leq C \hh_k,
\end{align}
where $H_k$ is the diameter of $\Omega\sMP$.  Given a face $F\sMP[k\ell]$  in $\Omega\sMP$ with diameter $H_{F\sMP[\ell k]}$, we denote its parameter domain representation as $\hf\sMP[k\ell]$. 
The meshsize on $F\sMP[k\ell]$ and $F\sMP[\ell k]$ is given by $h_{F\sMP[k\ell]}$ and $h_{F\sMP[k\ell]}$, respectively. Moreover, we assume for all $\ell\in\ifSet\sMP$ that $h_{F\sMP[k\ell]}\approx h_k$,  $h_{F\sMP[\ell k]}\approx h_\ell$ and $H_k \approx H_{F\sMP[k\ell]} \approx H_\ell$. Together with (\ref{equ:hath_h}), it follows that 
\begin{align}
\label{equ:hat_hkl_hkl}
 h_{k\ell}\approx H_k\hh_{k\ell}\approx H_\ell\hh_{k\ell}.
\end{align}
The introduced notation is illustrated in Figure\,\ref{fig:ParPhysMesh}.

\begin{figure}
 \includegraphics[width=\textwidth]{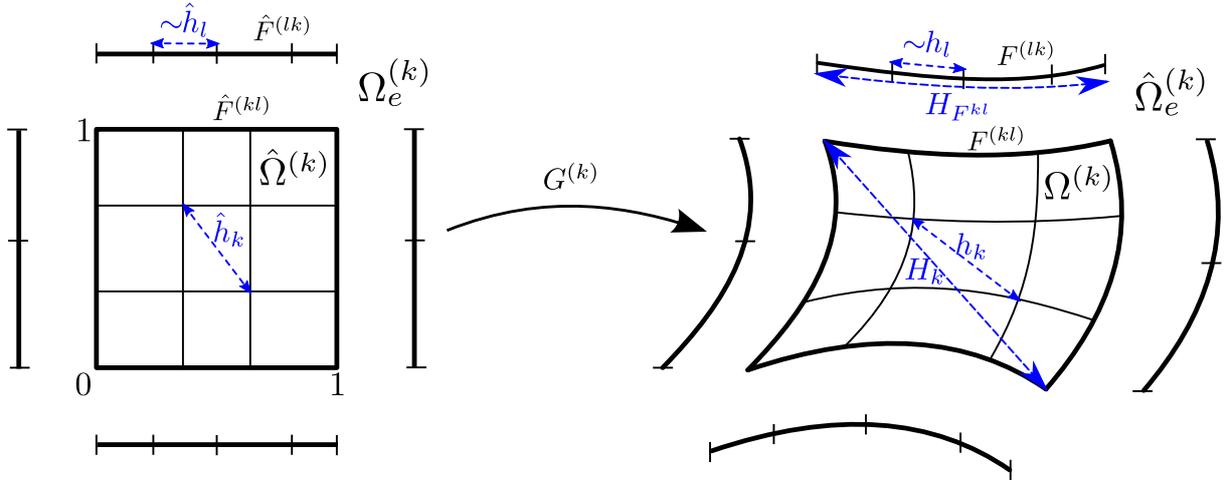}
 \caption{Illustration of the mesh in the parameter domain and in the physical domain, presenting the used notation.}
 \label{fig:ParPhysMesh}
\end{figure}

According to \cite{HL:VeigaChoPavarinoScacchi:2013a} and \cite{HL:HoferLanger:2016b}, we define a discrete $L^2$ norm and $H^1$ seminorm based on the coefficients $(c_i)_{i\in\iSet}$.   We denote by $c_{i,i^\iota-j}$ the coefficient corresponding to the basis function $\g{N}_{(i^1,\ldots,i^\iota-j,\ldots,i^d),p}$.
\begin{definition}
\label{def:discrete_norm}
Let $\g{u}\sMP\in V_{h,e}^{(k)}$ and $\p{u}\sMP$ its counterpart in the parameter domain. We define the  $L^2$ norm and $H^1$ seminorm for $\g{u}\sMP[k,k]$ as
     \begin{align*}
       |\p u\sMP[k,k]|_{\square}^2&:= \sum_{i\in\iSet\sMP[k,k]}|c_i^{(k,k)}|^2 \hh_k^2,\\
       |\p u\sMP[k,k]|_{\nabla}^2&:=\sum_{\iota=1}^d  |\p u|_{\xi^\iota}^2 \text{, where }\quad |\p u|_{\xi^\iota}^2:= \sum_{i\in\mathcal{I}_\iota^{(k,k)}}|c_{i,i^\iota}^{(k,k)} -c_{i,i^\iota-1}^{(k,k)}|^2,
     \end{align*} 
     where $i\in\mathcal{I}_\iota^{(k,k)}\subset\iSet\sMP[k,k]$ such that $c_{i,i^\iota-1}^{(k,k)}$ is well defined.
     
     Analogously, we define the discrete $L^2$ norm on $F\sMP[k\ell]$ and  $F\sMP[\ell k]$ via
     \begin{align*}
      |\p u\sMP[k,k]_{|\p F\sMP[k\ell]}|_{\square}^2&:= \sum_{i\in\iSet(F\sMP[k\ell])}|c_i^{(k,k)}|^2 \hh_k,\\
      |\p u\sMP[k,\ell]|_{\square}^2&:= \sum_{i\in\iSet\sMP[k,\ell]}|c_i^{(k,\ell)}|^2 \hh_\ell,
     \end{align*}
    where $\iSet(F\sMP[k\ell])\subset\iSet\sMP[k,k]$ are all indices of basis functions $\g{N}_{i,p}\sMP$ which have a support on $F\sMP[k\ell]\subset\overline{\Omega}\sMP$.
\end{definition}

\begin{assumption}
\label{ass:geoMap}
We assume that geometrical mapping $G\sMP$ has the properties 
\begin{align*}
	\|\nabla G\sMP\|_{L^\infty((0,1)^d)}\approx H_k\quad \text{ and }\quad\|\det\nabla  G\sMP\|_{L^\infty((0,1)^d)}\approx H_k^{d},
\end{align*}
where the hidden constants are independent of  $h_k$ and $H_k$.
\end{assumption}

\begin{proposition}
\label{prop:equivalence_l2_h1}
 Let $\g{u}\sMP\in V_{h,e}^{(k)}$ and $\p{u}\sMP$ its counterpart in the parameter domain. We have that
 \begin{align*}
    |\p u\sMP[k,k]|_{\square}^2& \approx \|\p u\sMP[k,k]\|_{L^2((0,1)^2)}^2 \approx H_k^{-2}\|\g{u}\sMP[k,k]\|_{L^2(\Omega\sMP)}^2 ,\\
    |\p u\sMP[k,k]|_{\nabla}^2  &\approx |\p u\sMP[k,k]|_{H^1((0,1)^2)}^2 \approx |\g{u}\sMP[k,k]|_{H^1(\Omega\sMP)}^2,\\
    |\p u\sMP[k,k]_{|\p F\sMP[k\ell]}|_{\square}^2 &\approx \|\p u\sMP[k,k]_{|\p{F}\sMP[k\ell]}\|_{L^2((0,1))}^2 \approx H_k^{-1}\|\g{u}\sMP[k,k]_{|F\sMP[k\ell]}\|_{L^2(F\sMP[k\ell])}^2,\\
    |\p u\sMP[k,\ell]|_{\square}^2 &\approx \|\p u\sMP[k,\ell]\|_{L^2((0,1))}^2 \approx H_k^{-1}\|\g{u}\sMP[k,\ell]\|_{L^2(F\sMP[k\ell])}^2,
 \end{align*}
 where the hidden constants do not depend on $h_k$ or $H_k$.
\end{proposition}

\begin{proof}
 Follows directly from Assumption~\ref{ass:geoMap}, Corollary 5.1. and Proposition 5.2. in \cite{HL:VeigaChoPavarinoScacchi:2013a} and the equivalence between of norms in the parameter and physical space, see, e.g., Lemma~3.5 in \cite{HL:BazilevsVeigaCottrellHughesSangalli:2006a}. It is important to note, that for the $H^1$ seminorm it holds that $|\g{u}\sMP|_{H^1(\Omega\sMP)}\approx H^{d-2} |u\sMP|_{H^1((0,1)^d)}$, which follows from the proof of Lemma~3.5 in \cite{HL:BazilevsVeigaCottrellHughesSangalli:2006a}.
\end{proof}
Next, we define the $L^2$-projection, in order to provide an approximation of $\g{u}\sMP[k,\ell]$ on $F\sMP[\ell k]$.
%
%
\begin{definition}
\label{def:L2_proj}
 We define by $\pi_{F\sMP[\ell k]}:V_{h}\sMP(F\sMP[k\ell])\to V\sMP_{h,e}(F\sMP[\ell k])$, the  orthogonal $L^2$-projection from the space $V_{h}\sMP(F\sMP[k\ell])$ onto $V_{h,e}^{(k)}(F^{(\ell k)})$.
 Moreover, for $\g{v}\in V_{h}\sMP(F\sMP[k\ell])$, we denote the coefficients of $\pi_{F\sMP[\ell k]}\g{v}$ by $\tilde{c}_i\sMP[k,\ell]$, i.e.,
 \begin{align}
 \label{equ:L2_proj_coeffs}
   \pi_{F\sMP[\ell k]}\g{v}= \sum_{i\in\iSet\sMP[k,\ell]}\tilde{c}_i\sMP[k,\ell] \g{N}_{i,p}\sMP[\ell].
 \end{align}

\end{definition}
\begin{lemma}
\label{lem:l2_projection}
 Let $\g{v}\in V_{h}\sMP$ and $\pi_{F\sMP[\ell k]}$ be the $L^2$-projection onto $V_{h,e}\sMP(F\sMP[\ell k])$ as in Definition~\ref{def:L2_proj}. Then it holds 
 \begin{align*}
  \|\g{v} - \pi_{F\sMP[\ell k]}\g{v}\|_{L^2(F\sMP[k\ell])}^2 \leq C h_\ell\frac{h_\ell}{h_k} \ho{|\g{v}|}_{H^1(\Omega\sMP)}^2,
 \end{align*}
 where the generic constant C is independent of $h_k$, $h_\ell$ or $H_k$.
\end{lemma}
\begin{proof}
Since the $L^2$-projection minimizes the error in the $L^2$ norm among all projections, we have that
\begin{align*}
\|\g{v} - \pi_{F\sMP[\ell k]}\g{v}\|_{L^2(F\sMP[k\ell])}^2 \leq C \|\g{v} - \mathcal{I}_{F\sMP[\ell k]}\g{v}\|_{L^2(F\sMP[k\ell])}^2,
\end{align*}
where $\mathcal{I}_{F\sMP[\ell k]}$ is the quasi B-Spline interpolant. By means of the interpolation estimate
\begin{align*}
 \|\g{v} - \mathcal{I}_{F\sMP[\ell k]}\g{v}\|_{L^2(F\sMP[k\ell])}^2 \leq C h_\ell^2\ho{|\g{v}|}_{H^1(F\sMP[k\ell])}^2
\end{align*}
and the discrete trace inequality, see, e.g., Lemma 4.3. in \cite{HL:EvansHughes:2013a},
\begin{align*}
 |\g{v}|_{H^1(F\sMP[k\ell])}^2 \leq C h_k^{-1}  |\g{v}|_{H^1(\Omega\sMP)}^2,
\end{align*}
we have
\begin{align*}
\|\g{v} - \pi_{F\sMP[\ell k]}\g{v}\|_{L^2(F\sMP[k\ell])}^2\leq C h_\ell^2|\g{v}|_{H^1(F\sMP[k\ell])}^2 \leq C h_\ell \frac{h_\ell}{h_k} |\g{v}|_{H^1(\Omega\sMP)}^2,
\end{align*}
which proves the estimate. \qed
\end{proof}
Now, we are in the position to define the discrete dG-norm and prove bounds in terms of $\|\cdot\|_{dG}$, as defined in (\ref{HL:dgNorm}).
\begin{definition}
\label{def:discrete_dg}
 Let $\g{u}\sMP\in V_{h,e}^{(k)}$ and $\p{u}\sMP$ its counterpart in the parameter domain. Moreover, for $\ell\in\ifSet\sMP$ let $\pi_{F\sMP[\ell k]}\g{u}\sMP[k,k]_{F\sMP[k\ell]}$, be the $L^2$-projection onto $V\sMP_{h,e}(F\sMP[\ell k])$ according to Definition~\ref{def:L2_proj} with coefficients $(\tilde{c}_i\sMP[k,\ell])_{i\in\iSet\sMP[k,\ell]}$ as in (\ref{equ:L2_proj_coeffs}). We define the discrete dG-norm $|\p{u}\sMP|_{dg}$ as
 \begin{align}
 \label{equ:discrete_dg}
  |\p u\sMP|_{dg}^2 := &|\p u\sMP[k,k]|_{\nabla}^2 + \sum_{\ell\in\ifSet\sMP} \frac{\delta }{\hh_{k\ell}} |\g u\sMP[k,\ell]-\pi_{F\sMP[\ell k]}\g u\sMP[k,k]_{|F\sMP[k\ell]}|_{\square}^2\\
		   =&|\p u\sMP[k,k]|_{\nabla}^2 + \sum_{\ell\in\ifSet\sMP} \frac{\delta }{\hh_{k\ell}} \sum_{i\in\iSet\sMP[k,\ell]}|c_i\sMP[k,\ell]-\tilde{c}_i\sMP[k,\ell]|^2 \hh_\ell,
 \end{align}
where $\delta$ is as in (\ref{HL:dgNorm}). We note that $|\g u\sMP[k,\ell]-\pi_{F\sMP[\ell k]}\g u\sMP[k,k]_{|F\sMP[k\ell]}|_{\square}^2$ is defined as $\sum_{i\in\iSet\sMP[k,\ell]}|c_i\sMP[k,\ell]-\tilde{c}_i\sMP[k,\ell]|^2 \hh_\ell$, cf. Definition\,\ref{def:discrete_norm} and (\ref{equ:coeff_dgSplineParameter}). 
\end{definition}
\begin{proposition}
\label{prop:equivalence_dg_disc}
 Let $\g{u}\sMP\in V_{h,e}^{(k)}$ and $\p{u}\sMP$ its counterpart in the parameter domain. Then we have
 \begin{align}
  \label{equ:discrete_dg_upper}
  H^{d-2}_k|\p u\sMP|_{dG}^2 \leq C \|\g{u}\sMP\|_{dG}^2,
 \end{align}
and 
\begin{align}
 \label{equ:discrete_dg_lower}
\ho { \|\g{u}\sMP\|_{dG}^2 \leq C\qh\sMP H^{d-2}_k |\p u\sMP|_{dG}^2 },
\end{align}
where $\qh\sMP:=\max_{\ell\in\ifSet\sMP}\left(\frac{h_\ell}{h_k}+\frac{h_\ell^2}{h_k^2}\right)$ and the generic constant is independent of $h_k$ and $H_k$.
\end{proposition}
\begin{proof}
 We first prove (\ref{equ:discrete_dg_upper}). The discrete dG-norm is defined as
 \begin{align*}
  |\p u\sMP|_{dg}^2 := |\p u\sMP[k,k]|_{\nabla}^2 + \sum_{\ell\in\ifSet\sMP}\frac{\delta}{\hh_{k\ell}} |\g u\sMP[k,\ell]-\pi_{F\sMP[\ell k]}\g u\sMP[k,k]_{|F\sMP[k\ell]}|_{\square}^2,
 \end{align*}
  where we can immediately bound the first term according to Proposition~\ref{prop:equivalence_l2_h1} by
  \begin{align}
   \label{equ:proof_discrete:upper:a}
   H^{d-2}_k|\p u\sMP[k,k]|_{\nabla}^2 \leq C |\g{u}\sMP[k,k]|_{H^1(\Omega\sMP)}^2.
  \end{align}
  For the second term, it holds
  \begin{align}
  \label{equ:proof_discrete:upper:b}
  \begin{split}
   H^{d-2}_k\sum_{\ell\in\ifSet\sMP}\frac{\delta }{\hh_{k\ell}}|\g u\sMP[k,\ell]-\pi_{F\sMP[\ell k]}\g u\sMP[k,k]_{|F\sMP[k\ell]}|_{\square}^2 \leq& \sum_{\ell\in\ifSet\sMP}\frac{\delta H^{d-1}_k }{h_{k\ell}}|\g u\sMP[k,\ell]-\pi_{F\sMP[\ell k]}\g u\sMP[k,k]_{|F\sMP[k\ell]}|_{\square}^2\\ 
   \leq& C\sum_{\ell\in\ifSet\sMP}\frac{\delta }{h_{k\ell}}\|\g{u}\sMP[k,\ell]-\pi_{F\sMP[\ell k]}\g{u}\sMP[k,k]_{|F\sMP[k\ell]}\|_{L^2(F\sMP[k\ell])}^2\\
   =& C\sum_{\ell\in\ifSet\sMP}\frac{\delta}{h_{k\ell}}\|\pi_{F\sMP[\ell k]}\left(\g{u}\sMP[k,\ell]-\g{u}\sMP[k,k]_{|F\sMP[k\ell]}\right)\|_{L^2(F\sMP[k\ell])}^2\\
   \leq& C\sum_{\ell\in\ifSet\sMP}\frac{\delta }{h_{k\ell}}\|\g{u}\sMP[k,\ell]-\g{u}\sMP[k,k]_{|F\sMP[k\ell]}\|_{L^2(F\sMP[k\ell])}^2,\\
  \end{split}
  \end{align}
  where we used (\ref{equ:hat_hkl_hkl}), Proposition~\ref{prop:equivalence_l2_h1}, the fact that $\pi_{F\sMP[\ell k]}\g{u}\sMP[k,\ell] = \g{u}\sMP[k,\ell]$ and the stability of the $L^2$-projection in the $L^2$ norm. Combining (\ref{equ:proof_discrete:upper:a}) and (\ref{equ:proof_discrete:upper:b}) gives 
  \begin{align*}
   H^{d-2}_k|\p u\sMP|_{dg}^2 \leq C \left(|\g{u}\sMP[k,k]|_{H^1(\Omega\sMP)}^2 + \sum_{\ell\in\ifSet\sMP}\frac{\delta}{h_{k\ell}}\|\g{u}\sMP[k,\ell]-\g{u}\sMP[k,k]_{|F\sMP[k\ell]}\|_{L^2(F\sMP[k\ell])}^2\right) = C \|\g{u}\sMP\|_{dG}^2,
  \end{align*}
  where C is a generic constant independent of $h$ and $H$.
  We now proof the the second estimate (\ref{equ:discrete_dg_lower}). The dG-norm reads
  \begin{align*}
   \|\g{u}\sMP\|_{dG}^2=|\g{u}\sMP|_{H^1(\Omega\sMP)}^2 + \sum_{\ell\in\ifSet\sMP}\frac{\delta}{h_{k\ell}}\|\g{u}\sMP[k,\ell]-\g{u}\sMP[k,k]_{|F\sMP[k\ell]}\|_{L^2(F\sMP[k\ell])}^2.
  \end{align*}
  Similar as before, we bound the first term by means of Proposition~\ref{prop:equivalence_l2_h1} via
    \begin{align}
   \label{equ:proof_discrete:lower:a}
   |\g{u}\sMP[k,k]|_{H^1(\Omega\sMP)}^2 \leq C H^{d-2}_k|\p u\sMP[k,k]|_{\nabla}^2.
  \end{align}
  For the second term, we have for $\ell\in\ifSet\sMP$
  \begin{align}
  \label{equ:proof_discrete:lower:b}
  \|\g{u}\sMP[k,\ell]-\g{u}\sMP[k,k]_{|F\sMP[k\ell]}\|_{L^2(F\sMP[k\ell])}^2 \leq \|\g{u}\sMP[k,\ell]-\pi_{F\sMP[\ell k]}\g{u}\sMP[k,k]_{|F\sMP[k\ell]}\|_{L^2(F\sMP[k\ell])}^2 +\|\pi_{F\sMP[\ell k]}\g{u}\sMP[k,k]_{|F\sMP[k\ell]}-\g{u}\sMP[k,k]_{|F\sMP[k\ell]}\|_{L^2(F\sMP[k\ell])}^2,
  \end{align}
  The first term of (\ref{equ:proof_discrete:lower:b}) can be estimated by means of Proposition~\ref{prop:equivalence_l2_h1} by
  \begin{align}
  \label{equ:proof_discrete:lower:ba}
   \|\g{u}\sMP[k,\ell]-\pi_{F\sMP[\ell k]}\g{u}\sMP[k,k]_{|F\sMP[k\ell]}\|_{L^2(F\sMP[k\ell])}^2 \leq C H^{d-1}_k |\g u\sMP[k,\ell]-\pi_{F\sMP[\ell k]}\g u\sMP[k,k]_{|F\sMP[k\ell]}|_{\square}^2.
  \end{align}
  Lemma~\ref{lem:l2_projection} yields for the second term in \eqref{equ:proof_discrete:lower:b}
  \begin{align*}
   \|\pi_{F\sMP[\ell k]}\g{u}\sMP[k,k]_{|F\sMP[k\ell]}-\g{u}\sMP[k,k]_{|F\sMP[k\ell]}\|_{L^2(F\sMP[k\ell])}^2 \leq C h_\ell\frac{h_\ell}{h_k} |\g{u}\sMP[k,k]|^2_{H^1(\Omega\sMP)},
  \end{align*}
  since $\g{u}\sMP[k,k]\in V\sMP_h$
  and according to Proposition~\ref{prop:equivalence_l2_h1}, we obtain
  \begin{align}
  \label{equ:proof_discrete:lower:bb}
    \|\pi_{F\sMP[\ell k]}\g{u}\sMP[k,k]_{|F\sMP[k\ell]}-\g{u}\sMP[k,k]_{|F\sMP[k\ell]}\|_{L^2(F\sMP[k\ell])}^2 \leq C h_\ell\frac{h_\ell}{h_k} H^{d-2}_k|\p{u}\sMP[k,k]|^2_{\nabla}.
  \end{align}
  Combining (\ref{equ:proof_discrete:lower:b}) with (\ref{equ:proof_discrete:lower:ba}) and (\ref{equ:proof_discrete:lower:bb}) and using the fact that $\frac{h_\ell}{h_{k\ell}}\frac{h_\ell}{h_k} = 2\left(\frac{h_\ell}{h_k}+\frac{h_\ell^2}{h_k^2}\right)$ together with $|\ifSet\sMP|\leq 4$ and (\ref{equ:hat_hkl_hkl}) gives
  \begin{align*}
   \sum_{\ell\in\ifSet\sMP}\frac{\delta}{h_{k\ell}}\|\g{u}\sMP[k,\ell]-\g{u}\sMP[k,k]_{|F\sMP[k\ell]}\|_{L^2(F\sMP[k\ell])}^2 \leq C \biggl( &8\delta\max_{\ell\in\ifSet\sMP}\left(\frac{h_\ell}{h_k}+\frac{h_\ell^2}{h_k^2}\right) H^{d-2}_k|\p{u}\sMP[k,k]|^2_{\nabla}\\
   &+ \sum_{\ell\in\ifSet\sMP} \frac{\delta H^{d-2}_k}{\hh_{k\ell}}|\g u\sMP[k,\ell]-\pi_{F\sMP[\ell k]}\g u\sMP[k,k]_{|F\sMP[k\ell]}|_{\square}^2\biggr),
  \end{align*}
which concludes the proof. \qed
\end{proof}

We now provide properties of the local index spaces. Since we consider only the two dimensional problem, we can interpret the coefficients $(c_i\sMP)_{i\in\iSet_e\sMP}$ of $\g{u}\sMP$ as a matrix plus four additional vectors for the extra boundary, i.e., $\boldsymbol C_e\sMP:=\{\boldsymbol C\sMP[k,k], \{\boldsymbol c\sMP[k,\ell]\}_{\ell\in\ifSet\sMP}\}\in\mathcal{R}_e\sMP:=\mathbb{R}^{M_1\sMP\times M_2\sMP}+\prod_{\ell\in\ifSet\sMP} \mathbb{R}^{M\sMP[\ell k]}$, where $\boldsymbol C\sMP[k,k]:=(c_i\sMP)_{i\in\iSet\sMP[k,k]}$ and $\boldsymbol c\sMP[k,\ell]:=(c_i)_{i\in\iSet\sMP[k,\ell]}$ for $\ell\in\ifSet\sMP$.
The number $M\sMP[\ell k]$ denotes the number of coefficients associated to $\p{F}\sMP[\ell k]$, i.e., $M\sMP[\ell k]= |\iSet\sMP[k,\ell]|$. We note, that there exists a $\iota^*\in\{1,\ldots,d\}$ such that $M\sMP[\ell k] = M_{\iota^*}\sMP[\ell]$. Moreover, we assume that there exists a constant $\beta\in\mathbb{R}^{{+}}$ such that ${\beta}^{-1}M_2\sMP \leq M_1\sMP\leq \beta M_2\sMP.$

The entries of the matrix $\boldsymbol{C}_e\sMP$ can be interpreted as values on a uniform grid $\mathcal{T}\sMP_e:=\mathcal{T}\sMP\cup\bigcup_{\ell\in\ifSet\sMP}\mathcal{T}\sMP[k\ell]$ on $\overline{\p{\Omega}_e}$, where $\mathcal{T}\sMP$ and $\mathcal{T}\sMP[k\ell]$ are the grids corresponding to $\boldsymbol{C}\sMP[k,k]$ and $\boldsymbol{c}\sMP[k,\ell]$, respectively.  The meshes $\mathcal{T}\sMP$ and $\mathcal{T}\sMP[k\ell]$ have a characteristic meshsize 
\begin{align*}
 \tth_k:= \left(\frac{1}{(M_1\sMP-1)^2}+\frac{1}{(M_2\sMP-1)^2}\right)^{1/2} \text{ and }  \tth\sMP[\ell k]=\frac{1}{M\sMP[\ell k]-1},
\end{align*}
 respectively. Hence, we have
\begin{align*}
 C^{-1}_{\beta,k} \frac{1}{M_1\sMP} \leq \tth_k \leq C_{\beta,k} \frac{1}{M_1\sMP}, \quad \text{and}\quad  C^{-1}_{\beta,l} \tth_\ell \leq \tth\sMP[\ell k] \leq C_{\beta,l} \tth_\ell,
\end{align*}
where the constants $C_{\beta,k}$ and $C_{\beta,l}$ depend only on $\beta$. 
By basic properties of the geometrical mapping $G$ and the B-Splines $\p{N}_{i,p}$, it is easy to see that
\begin{align}
\label{equ:h/H}
C^{-1}_{G,\beta}\tth_k\leq h_k/H_k \leq \tth_k C_{G,\beta} \text{ and } C^{-1}_{\beta} \tth_k\leq \hh_k \leq C_{\beta} \tth_k,
\end{align}
where the 
the constant $C_{\beta}$ depends only on $\beta$ and the constant $C_{G,\beta}$ additionally also on $G$.
Finally, we define the harmonic average $\tth_{k\ell} := 2\tth_k\tth_\ell/(\tth_k+\tth_\ell)$.

We are now able to introduce a dG-norm on the discrete coefficient-space $\mathcal{R}_e\sMP$ as follows
\begin{align}
\tl \boldsymbol C_e\sMP \tl_{dG}^2:= \tl \boldsymbol C\sMP[k,k] \tl_{\nabla}^2  + \sum_{\ell\in\ifSet\sMP}\frac{\delta}{ \tth_{k\ell}}\sum_{i=1}^{M\sMP[\ell k]}|c_i\sMP[k,\ell]-\tilde{c}_i\sMP[k,\ell]|^2 \tth_{l},
\end{align}
where $\tilde{c}_i\sMP[k,\ell]$ is defined analogously as in Definition~\ref{def:discrete_dg} and 
\begin{align*}
 \tl \boldsymbol C\sMP[k,k] \tl_{\nabla}^2:=\sum_{\iota=1}^2 \sum_{\substack{i=1 \\i^\iota=2}}^{M\sMP}|c_{i,i^\iota}\sMP[k,k]-c_{i,i^\iota-1}\sMP[k,k]|^2.
\end{align*}
We note, that for given function $\g{u}\in V_{h,e}\sMP$ with coefficient representation $\boldsymbol C_e\in \mathcal{R}_e\sMP$, we have 
\begin{align}
\label{equ:discreteCoeff}
 C^{-1}|\p u|_{dG}^2 \leq \tl\boldsymbol C_e\tl_{dG}^2 \leq C |\p u|_{dG}^2,
\end{align}
where the constant $C$ depends only on the constants $C_{G,\beta}$ and $C^{-1}_{\beta,k}$ from patch $k$ and all its neighbouring patches.

This motivates the definition of an operator  $(\cdot)_I:\, C(\overline{\p{\Omega}_e}\sMP) \to \mathcal{R}_e\sMP$, where $C(\overline{\p{\Omega}_e}\sMP):=C(\overline{\p{\Omega}}\sMP)+\prod_{i=1}^4C(\overline{\hf\sMP[k\ell]})$, which evaluates a continuous functions on $\overline{\p{\Omega}_e\sMP}$ in the grid points $x_i$ of $\mathcal{T}_e$. Moreover, we introduce an operator  $\chi\sMP: \mathcal{R}_e\sMP \to H^1(\p{\Omega}_e):=H^1(\p{\Omega})+\prod_{\ell\in\ifSet\sMP}H^1(\hf\sMP[\ell k])$,
 that provides a piecewise bilinear interpolation of the given grid values, i.e., $\chi\sMP(\boldsymbol{v})\in \mathcal{Q}_1(\mathcal{T}_e\sMP):= \mathcal{Q}_1(\mathcal{T}\sMP)+\prod_{\ell\in\ifSet\sMP}\mathcal{P}_1(\mathcal{T}\sMP[k\ell])$. Here $\mathcal{Q}_1(\mathcal{T}\sMP)$ is the space of piecewise bilinear functions on $\mathcal{T}\sMP$ and $\mathcal{P}_1(\mathcal{T}\sMP[k\ell])$ the space of piecewise linear functions on $\mathcal{T}\sMP[k\ell]$.

 Given values on an edge $\hf\sMP[k\ell]_e:=\hf\sMP[k\ell]\cup \hf\sMP[\ell k]$ and its associated grid $\mathcal{T}\sMP[k\ell]_e:= \mathcal{T}\sMP[k]_{|\hf\sMP[k\ell]} \cup \mathcal{T}\sMP[\ell,k]$, we need to define its linear interpolation and a discrete harmonic extension to the interior.
  In order to do so, let us denote  all indices of grid points $x_i$ associated to $\hf\sMP[k\ell]_e$ by $\mathcal{I}(\hf\sMP[k\ell]_e)$. Additionally, let $\mathcal{P}_1(\mathcal{T}\sMP[k\ell]_e):=\mathcal{P}_1(\mathcal{T}\sMP[k\ell])+\mathcal{P}_1(\mathcal{T}\sMP[\ell k])$ be the space of piecewise linear spline functions on $\mathcal{T}\sMP[k\ell]_e$.   We define the interpolation of values on $\hf\sMP[k\ell]_e$ via the restriction of the operator $\chi\sMP$ to $\hf\sMP[k\ell]_e$, denoted by $\chi_{\hf\sMP[k\ell]_e}\sMP : \mathbb{R}^{{M_\iota\sMP+M\sMP[\ell k]}} \to H^1(\hf\sMP[k\ell])+ H^1(\hf\sMP[\ell k])$ with an analogous definition. In a similar way, we define the interpolation operator for the whole boundary $\Gamma_e$, denoted by $\chi_{\Gamma,e}\sMP: \mathbb{R}^{{|\mathcal{I}(\Gamma_e)|}} \to H^1(\partial\p{\Omega}\sMP)+\prod_{\ell\in\ifSet\sMP}H^1(\hf\sMP[\ell k])$, where $\mathcal{I}(\Gamma_e): = \{i: x_i \in \Gamma_e\}$.
  %
  %
   According to \cite{HL:VeigaChoPavarinoScacchi:2012a}, we define a seminorm for grid points on an edge $\hf\sMP[k\ell]$ via the interpolation to functions in  $\mathcal{P}_1(\mathcal{T}\sMP[k\ell]_e)$:
\begin{definition}
Let $\hf\sMP[k\ell]$ be an edge of $\p{\Omega}\sMP$ along dimension $\iota$.
Then we define the 
seminorm
$\tl\boldsymbol{v}\tl_{\hf\sMP[k\ell]} := 
|\chi\sMP_{\hf\sMP[k\ell]}(\boldsymbol{v})|_{H^{1/2}(\hf\sMP[k\ell])}$
for all $\boldsymbol{v}\in\mathbb{R}^{\mathcal{I}(\hf\sMP[k\ell])}$.
%
%
  \end{definition}

   \begin{definition}
   Let $\mathcal{H}_{\mathcal{Q}_{1,e}}\sMP$ be the standard discrete harmonic extension in the sense of $a_e\sMP(\cdot,\cdot)$ into the piecewise bilinear space $\mathcal{Q}_{1,e}$, see \cite{HL:DryjaGalvisSarkis:2013a} for a formal definition. This defines the lifting operator $\boldsymbol{H}_e\sMP : \mathbb{R}^{|\mathcal{I}(\Gamma_e)|} \to \mathcal{R}_e\sMP$ by
   \begin{align*}
   \begin{split}
             \boldsymbol{b}\mapsto \boldsymbol{H}_e\sMP(\boldsymbol{b}):= (\mathcal{H}_{\mathcal{Q}_{1,e}}\sMP(\chi_{\Gamma,e}\sMP(\boldsymbol{b})))_I.
   \end{split}
   \end{align*}
  \end{definition}
\begin{theorem}
\label{HL:thm:realMatProp}
   Let $\hf\sMP[k\ell]$ be a particular side of $\partial\p{\Omega}\sMP$ and the constant $\beta\in\mathbb{R}^{{+}}$ such that
    ${\beta^{-1}}M_2\sMP \leq M_1\sMP\leq \beta M_2\sMP.$
   Then the following statements hold:
   \begin{enumerate}
    \item For all $\boldsymbol{b}\in \mathbb{R}^{|\mathcal{I}(\Gamma_e)|}$ that vanish on the twelve components corresponding to the twelve corners $\mathcal{V}_e\sMP$, 
    the estimate
    \begin{align*}    
     \tl\boldsymbol{H}_e\sMP(\boldsymbol{b})\tl_{dG}^2 \leq C  \tqh\sMP  (1+\log^2\tth_k^{-1})\sum_{\ell\in\ifSet\sMP}\left(\tl \boldsymbol b|_{\hf\sMP[k\ell]}\tl_{\hf\sMP[k\ell]}^2 +\frac{\delta }{\tth_{k\ell}}\sum_{i=1}^{M\sMP[\ell k]}|b_i\sMP[k,\ell]-\tilde{b}_i\sMP[k,\ell]|^2 \tth_\ell \right), 
    \end{align*}
    holds, where $\tqh\sMP:=\max_{\substack{\ell\in\ifSet\sMP}}\left(\frac{\tth_\ell}{\tth_k}+\frac{\tth_\ell^2}{\tth_k^2}\right)$ and the constant $C$ does not depend on $h_k$ or $H_k$.
    \item 
    The estimate
    \begin{align*}
      \tl \boldsymbol{C} \tl^2_{dG} \geq C \left(\tl\boldsymbol{C}|_{\hf\sMP[k\ell]}\tl^2_{\hf\sMP[k\ell]} +  \sum_{\ell\in\ifSet\sMP}
      \frac{\delta }{\tth_{k\ell}}\sum_{i=1}^{M\sMP[\ell k]}|c_i\sMP[k,\ell]-\tilde{c}_i\sMP[k,\ell]|^2 \tth_\ell\right)
    \end{align*}
    is valid for all $\boldsymbol{C} \in\mathcal{R}\sMP_e$,
    where the constant $C$ does not depend on $h_k$ or $H_k$.
   \end{enumerate}
\end{theorem}
\begin{proof}
 For a better readability, we will omit the superscript $(k)$.
 
 The discrete dG-norm for matrices is defined as
 \begin{align*}
  \tl\boldsymbol{H}_e(\boldsymbol{b})\tl_{dG}^2=\tl\boldsymbol{H}_e(\boldsymbol{b})\tl_{\nabla}^2 + \sum_{\ell\in\ifSet}\frac{\delta }{\tth_{k\ell}}\sum_{i=1}^{M\sMP[\ell k]}|b_i\sMP[k,\ell]-\tilde{b}_i\sMP[k,\ell]|^2 \tth_\ell.
 \end{align*}
 By means of a similar estimate for piecewise bilinear functions as in Proposition~\ref{prop:equivalence_l2_h1}, we have the estimates
  \begin{align*}
   \tl\boldsymbol{H}_e(\boldsymbol{b})\tl_{\nabla}^2 &= |\mathcal{H}_{\mathcal{Q}_{1,e}}\left(\chi_{\Gamma,e}(\boldsymbol{b})\right)|_{\nabla}^2 \leq C|\mathcal{H}_{\mathcal{Q}_{1,e}}\left(\chi_{\Gamma,e}(\boldsymbol{b})\right)|^2_{H^1(\p{\Omega})},\\
   \sum_{i=1}^{M\sMP[\ell k]}|b_i\sMP[k,\ell]-\tilde{b}_i\sMP[k,\ell]|^2 \tth_\ell^2 &= |\chi_{\hf\sMP[\ell k]}(\boldsymbol{b})-\pi_{\hf\sMP[\ell k]}\chi_{\hf\sMP[k\ell]}(\boldsymbol{b})|_{\square}^2\leq C\|\chi_{\hf\sMP[\ell k]}(\boldsymbol{b})-\pi_{\hf\sMP[\ell k]}\chi_{\hf\sMP[k\ell]}(\boldsymbol{b})\|_{L^2(\hf\sMP[k\ell])}^2\\
   &\leq C\|\chi_{\hf\sMP[\ell k]}(\boldsymbol{b})-\chi_{\hf\sMP[k\ell]}(\boldsymbol{b})\|_{L^2(\hf\sMP[k\ell])}^2,
  \end{align*}
and, therefore,
\begin{align*}
 \tl\boldsymbol{H}_e(\boldsymbol{b})\tl_{dG}^2&\leq C\left( |\mathcal{H}_{\mathcal{Q}_{1,e}}\left(\chi_{\Gamma,e}(\boldsymbol{b})\right)|^2_{H^1(\p{\Omega})}+ \sum_{\ell\in\ifSet}\frac{\delta}{\tth_{k\ell}}\|\chi_{\hf\sMP[\ell k]}(\boldsymbol{b})-\chi_{\hf\sMP[k\ell]}(\boldsymbol{b})\|_{L^2(\hf\sMP[k\ell])}^2\right)\\
 &=C \|\mathcal{H}_{\mathcal{Q}_{1,e}}(\chi_{\Gamma,e}(\boldsymbol{b}))\|^2_{dG}.
\end{align*}
%
Using the FE equivalent of Lemma \ref{lem:equDiscHarmonic}, see, e.g. \cite{HL:DryjaGalvisSarkis:2007a}, we can estimate $\|\mathcal{H}_{\mathcal{Q}_{1,e}}(\chi_{\Gamma,e}(\boldsymbol{b}))\|^2_{dG}\leq C\|\mathcal{H}_{\mathcal{Q}_{1}}(\chi_{\Gamma,e}(\boldsymbol{b}))\|^2_{dG}$, where  the constant $C$ is independent of $h_k, H_k$ and $\delta$, and $\mathcal{H}_{\mathcal{Q}_{1}}$ is the standard discrete harmonic extension in the sense of $a\sMP(\cdot,\cdot)$. Hence, we obtain
\begin{align}
\label{equ:proof_MatProp_a}
 \|\mathcal{H}_{\mathcal{Q}_{1,e}}(\chi_{\Gamma,e}(\boldsymbol{b}))\|^2_{dG} &\leq C \left(|\mathcal{H}_{\mathcal{Q}_{1}}\left(\chi_{\Gamma,e}(\boldsymbol{b})\right)|^2_{H^1(\p{\Omega})}+ \sum_{\ell\in\ifSet}\frac{\delta }{\tth_{k\ell}}\|\chi_{\hf\sMP[\ell k]}(\boldsymbol{b})-\chi_{\hf\sMP[k\ell]}(\boldsymbol{b})\|_{L^2(\hf\sMP[k\ell])}^2\right).
\end{align}

%
For the second term of (\ref{equ:proof_MatProp_a}), we use the estimate
\begin{align}
\label{equ:proof_MatProp_aa}
\begin{split}
 \|\chi_{\hf\sMP[\ell k]}(\boldsymbol{b})-\chi_{\hf\sMP[k\ell]}(\boldsymbol{b})\|_{L^2(\hf\sMP[k\ell])}^2 \leq& \|\chi_{\hf\sMP[\ell k]}(\boldsymbol{b})-\pi_{\hf\sMP[\ell k]}\chi_{\hf\sMP[k\ell]}(\boldsymbol{b})\|_{L^2(\hf\sMP[k\ell])}^2 \\
 &+\|\pi_{\hf\sMP[\ell k]}\chi_{\hf\sMP[k\ell]}(\boldsymbol{b})-\chi_{\hf\sMP[k\ell]}(\boldsymbol{b})\|_{L^2(\hf\sMP[k\ell])}^2,
 \end{split}
\end{align}
where the first term can be estimated by FE equivalent of Proposition~\ref{prop:equivalence_l2_h1} with
\begin{align}
 \label{equ:proof_MatProp_aaa}
 \begin{split}
 \|\chi_{\hf\sMP[\ell k]}(\boldsymbol{b})-\pi_{\hf\sMP[\ell k]}\chi_{\hf\sMP[k\ell]}(\boldsymbol{b})\|_{L^2(\hf\sMP[k\ell])}^2 &\leq C |\chi_{\hf\sMP[\ell k]}(\boldsymbol{b})-\pi_{\hf\sMP[\ell k]}\chi_{\hf\sMP[k\ell]}(\boldsymbol{b})|_{\square}^2\\
 & = C \sum_{i=1}^{M\sMP[\ell k]}|b_i\sMP[k,\ell]-\tilde{b}_i\sMP[k,\ell]|^2 \tth_{l}.
 \end{split}
\end{align}
Since $\chi_{\hf\sMP[k\ell]}(\boldsymbol{b})$ is a piecewise linear function, we can use already existent estimates for the $L^2$-projection. 
\ho{We use the following estimate to bound the second term of (\ref{equ:proof_MatProp_aa})
\begin{align}
 \label{equ:proof_MatProp_aab}
 \|\pi_{\hf\sMP[\ell k]}\chi_{\hf\sMP[k\ell]}(\boldsymbol{b})-\chi_{\hf\sMP[k\ell]}(\boldsymbol{b})\|_{L^2(\hf\sMP[k\ell])}^2 \leq C \tth_\ell \frac{\tth_\ell}{\tth_k} |\mathcal{H}_{\mathcal{Q}_{1}}\left(\chi_{\Gamma,e}(\boldsymbol{b})\right)|_{H^1(\p{\Omega})}^2,
\end{align}
which follows by repeating the same arguments as in the proof of Lemma~\ref{lem:l2_projection} for bilinear functions. }
Combining inequalities (\ref{equ:proof_MatProp_aaa}) and (\ref{equ:proof_MatProp_aab}) with (\ref{equ:proof_MatProp_aa}) and using it in (\ref{equ:proof_MatProp_a}) gives
 \begin{align}
 \label{equ:proof_MatProp_b}
 \begin{split}
 \|\mathcal{H}_{\mathcal{Q}_{1}}(\chi_{\Gamma,e}(\boldsymbol{b}))\|^2_{dG} \leq C \bigl( \delta \tqh&|\mathcal{H}_{\mathcal{Q}_{1}}\left(\chi_{\Gamma,e}(\boldsymbol{b})\right)|^2_{H^1(\p{\Omega})} \\
 &+ \sum_{\ell\in\ifSet}\frac{\delta}{\tth_{k\ell}}\sum_{i=1}^{M\sMP[\ell k]}|b_i\sMP[k,\ell]-\tilde{b}_i\sMP[k,\ell]|^2 \tth_\ell\bigr).
 \end{split}
\end{align}

We are now in the position to use the available theory for the standard discrete harmonic extension $\mathcal{H}_{\mathcal{Q}_{1}}$ to estimate the first term of (\ref{equ:proof_MatProp_b}). Recalling the estimate 
\begin{align*}
 |\mathcal{H}_{\mathcal{Q}_1}\left(\chi_{\Gamma,e}(\boldsymbol{b})\right)|_{H^1(\p{\Omega})}^2 \leq C(1+\log^2 \tth^{-1}_k) \sum_{\ell\in\ifSet}|\chi_{\hf\sMP[k\ell]}(\boldsymbol{b})|_{H^{1/2}(\hf\sMP[k\ell])}^2,
\end{align*}
see Theorem. 5 in \cite{HL:MandelDohrmann:2003a} or the proof of Theorem 5.1. in \cite{HL:VeigaChoPavarinoScacchi:2013a}, and
since $|\chi_{\hf\sMP[k\ell]}(\boldsymbol{b})|_{H^{1/2}(\hf\sMP[k\ell])}^2= \tl\boldsymbol{b}|_{\hf\sMP[k\ell]}\tl^2_{\hf\sMP[k\ell]}$, we obtain
\begin{align*}
 \tl\boldsymbol{H}_e(\boldsymbol{b})\tl_{dG}^2 \leq C \tqh\sMP (1+\log^2 \tth_k^{-1})\sum_{\ell\in\ifSet}\left(\tl \boldsymbol{b}|_{\hf\sMP[k\ell]}\tl_{\hf\sMP[k\ell]}^2 + \frac{\delta}{\tth_{k\ell}}\sum_{i=1}^{M\sMP[\ell k]}|b_i\sMP[k,\ell]-\tilde{b}_i\sMP[k,\ell]|^2 \tth_\ell \right).
\end{align*}
This proves the first inequality.  Again, by means of a similar estimate for piecewise bilinear functions as in Proposition~\ref{prop:equivalence_l2_h1} and according to Theorem 5.1(b) in \cite{HL:VeigaChoPavarinoScacchi:2013a}, we have
\begin{align*}
  \tl \boldsymbol{C}\sMP[k,k]\tl_{\nabla}^2 \geq C|\chi(\boldsymbol{C}\sMP[k,k])|^2_{H^1(\p{\Omega})} \geq C \tl\boldsymbol{C}|_{\hf\sMP[k\ell]}\tl^2_{\hf\sMP[k\ell]}.
\end{align*}
Therefore, we obtain
\begin{align*}
 \tl \boldsymbol{C} \tl^2_{dG}  &=  \tl \boldsymbol{C}\sMP[k,k]\tl_{\nabla}^2 +  \sum_{\ell\in\ifSet}\frac{\delta }{\tth_{k\ell}}\sum_{i=1}^{M\sMP[\ell k]}|c_i\sMP[k,\ell]-\tilde{c}_i\sMP[k,\ell]|^2 \tth_{l}\\
 &\geq C\biggl(\tl\boldsymbol{C}|_{\hf\sMP[k\ell]}\tl^2_{\hf\sMP[k\ell]} +  \sum_{\ell\in\ifSet}\frac{\delta }{\tth_{k\ell}}\sum_{i=1}^{M\sMP[\ell k]}|c_i\sMP[k,\ell]-\tilde{c}_i\sMP[k,\ell]|^2 \tth_{l}\biggr),
\end{align*}
which concludes the proof.
\qed
\end{proof}

\section{Condition number bound}
\label{sec:cond_bound}
The goal of this section is to establish the condition number bound for $ M_{BDDC}^{-1}\widehat{S}$. 
Following \cite{HL:VeigaChoPavarinoScacchi:2013a}, we assume  that the mesh is quasi-uniform on each subdomain and the diffusion coefficient
is globally constant. Moreover, in \cite{HL:VeigaChoPavarinoScacchi:2013a} one can also find a formal definition of the BDDC preconditioner $M_{BDDC}^{-1}$. It was already pointed out that the spectrum of $M_{BDDC}^{-1}\widehat{S}$ is equal to $M_{sD}^{-1}F$ up to zeros and ones.
For simplicity, we focus on a patch $\Omega_e^{(k)}$, with $k\in\{1,\ldots, N\}$, which does not touch the boundary $\partial\Omega$.

Let $\g{u}^{(k)}\in V_{h,e}^{(k)}$, then $\g{u}^{(k)}$ is determined by its coefficients $c_i^u, i\in\iSet$, which can be interpreted as a matrix $\boldsymbol C_e\sMP:=\{\boldsymbol C\sMP[k,k], \{\boldsymbol C\sMP[k,\ell]\}_{\ell\in\ifSet\sMP}\}\in\mathcal{R}_e\sMP$. In a similar way, we can identify functions on the trace space $W^{(k)}$. Finally, let $W_{\Delta}^{(k)}\subset W^{(k)}$ be the space of spline functions which vanish on the primal variables, i.e., in the corner points.
The following theorem provides an abstract estimate of the condition number using the coefficient scaling, cf. Theorem 6.1 in \cite{HL:VeigaChoPavarinoScacchi:2012a}:
\begin{theorem}
\label{HL:thm:abstractBDDC}
 Let the counting function ${\delta^{\dagger}}^{(k)}$ be chosen accordingly to the coefficient scaling strategy. Assume that there exist two positive constants $c_*,c^*$ and a boundary seminorm $|\cdot|_{W^{(k)}}$ on $W^{(k)}$, $k=1,\ldots, N$, such that
 \begin{align}
  \label{HL:equ:abstractBDDC_1}
  |\g{w}^{(k)}|_{W^{(k)}}^2 &\leq c^* s_e^{(k)}(\g{w}^{(k)},\g{w}^{(k)}) \quad\forall \g{w}^{(k)}\in W^{(k)},\\
  \label{HL:equ:abstractBDDC_2}
  |\g{w}^{(k)}|_{W^{(k)}}^2 &\geq c_* s_e^{(k)}(\g{w}^{(k)},\g{w}^{(k)}) \quad\forall \g{w}^{(k)}\in W_{\Delta}^{(k)},\\
  \label{HL:equ:abstractBDDC_3}
  |\g{w}^{(k)}|_{W^{(k)}}^2 &= \sum_{\ell\in\ifSet\sMP}|\g{w}^{(k)}|_{F\sMP[k\ell]_e}|_{W^{(k\ell)}}\quad\forall \g{w}^{(k)}\in W^{(k)},
 \end{align}
 where $|\cdot|_{W^{(k\ell)}}$ is a seminorm associated to the edge spaces $W^{(k)}|_{F\sMP[k\ell]_e}$ with $\ell\in\ifSet\sMP$. Then the condition number of the preconditioned BDDC operator $M_{BDDC}^{-1}\widehat{S}$ satisfies the bound
 \begin{align*}
  \kappa(M_{BDDC}^{-1}\widehat{S})\leq C(1+c_*^{-1} c^*),
 \end{align*}
where the constant $C$ is independent of $h$ and $H$.
\end{theorem}
Using this abstract framework, we obtain  the following condition number estimate for the BDDC preconditioner.
\begin{theorem}
\label{HL:thm:CondCoeffScal}
There exists a boundary seminorm such that the constants $c_*$ and $c^*$ of Theorem\,\ref{HL:thm:abstractBDDC} are bounded by 
 \begin{align*}
  c^* \leq C_1  \; \text{ and } \;
  c_*^{-1} \leq  C_2 \max_{\substack{1\leq k\leq N\\\ell\in\ifSet\sMP}}\left(\frac{h_\ell}{h_k}+\frac{h_\ell^2}{h_k^2}\right)^2 \max_{1\leq k\leq N} \left(1+\log^2\left(\frac{H_k}{h_k}\right)\right),
 \end{align*}
 where the constants $C_1$ and $C_2$ are independent of $H$ and $h$ . Therefore, the condition number of the isogeometric preconditioned BDDC operator is bounded by
 %
 \begin{align*}
  \kappa(M_{BDDC}^{-1}\widehat{S}) \leq C \max_{\substack{1\leq k\leq N\\\ell\in\ifSet\sMP}}\left(\frac{h_\ell}{h_k}+\frac{h_\ell^2}{h_k^2}\right)^2 \max_{1\leq k\leq N} \left(1+\log^2\left(\frac{H_k}{h_k}\right)\right),
 \end{align*}
 where the constant $C$ is independent of $H$ and $h$.
\end{theorem}
\begin{proof}
 The first step is to appropriately define the seminorm $|\cdot|_{W^{(k)}}^2$ in $W^{(k)}$:
 \begin{align*}
 \begin{split}
  |\g{w}^{(k)}|_{W^{(k)}}^2&:= \sum_{\ell\in\ifSet\sMP}|\g{w}^{(k)}|_{F^{(k\ell)}_e}|_{W^{(k\ell)}}^2,\\
  |\g{w}^{(k)}|_{F\sMP[k\ell]_e}|_{W^{(k\ell)}}^2 &:=
                        H^{d-2}_k\biggl(\tl\g{w}^{(k)}|_{F\sMP[k\ell]}\tl_{F\sMP[k\ell]}^2 
                     + |\g{w}\sMP|_{F\sMP[k\ell]}|_{\nabla}^2
                     +  \frac{\delta }{\tth_{k\ell}}|\g{w}\sMP[k,\ell]-\pi_{F\sMP[\ell k]}\g{w}\sMP_{|F\sMP[k\ell]}|_\square^2\biggr),\\
  \end{split}
  \end{align*}
  where $|\g{w}\sMP|_{F\sMP[k\ell]}|_{\nabla}^2$ has to be understood as the restriction of the discrete seminorm to $F\sMP[k\ell]$, cf., Definition~\ref{def:discrete_norm}, this essentially gives the differences along $F\sMP[k\ell]$. Furthermore, we define 
$\tl \g{w}^{(k)}|_{F\sMP[k\ell]} \tl_{F\sMP[k\ell]}:= \tl\mathbf{c}\tl_{F\sMP[k\ell]}$,
where $\mathbf{c}$  are the values $(c_i^w)_{i\in \iSet(F\sMP[k\ell])}$ written as a vector.

Given $\g{w}\sMP\in W\sMP$ 
we define its NURBS harmonic extension by  $\g{u}\sMP=\mathcal{H}_e\sMP(\g{w}\sMP)$ with coefficients $\boldsymbol C_{e}\sMP:=\{\boldsymbol C\sMP[k,k], \{\boldsymbol c\sMP[k,\ell]\}_{\ell\in\ifSet\sMP}\}$. Consider a single edge $F\sMP[k\ell]$, since $\tl \g{w}^{(k)}|_{F\sMP[k\ell]}\tl_{F\sMP[k\ell]}^2 = \tl \boldsymbol C|_{\hf\sMP[k\ell]}\sMP[k,k]  \tl^2_{\hf\sMP[k\ell]}$ and $|\g{w}\sMP[k,\ell]-\pi_{F\sMP[\ell k]}\g{w}\sMP_{|F\sMP[k\ell]}|_\square^2 = \sum_{i\in\iSet\sMP[k,\ell]}|c_i\sMP[k,\ell]-\tilde{c}_i\sMP[k,\ell]|^2 \hh_\ell$ we can estimate 
\begin{align*}
 \tl \g{w}^{(k)}|_{F\sMP[k\ell]}\tl_{F\sMP[k\ell]}^2 + 
 \frac{\delta }{\tth_{k\ell}}|\g{w}\sMP[k,\ell]-\pi_{F\sMP[\ell k]}\g{w}\sMP_{|F\sMP[k\ell]}|_\square^2 \leq C \tl \boldsymbol C\sMP_e \tl_{dG}^2,
\end{align*}
by means of Theorem~\ref{HL:thm:realMatProp}(b) and \eqref{equ:h/H}. Moreover, the second term of $| \cdot |^2_{W^{(k\ell)}}$ is a part of $\tl \cdot \tl_{dG}^2$, hence, 
we obtain the inequality
\begin{align*}
 |\g{w}^{(k)}|_{F\sMP[k\ell]}|_{W^{(k\ell)}}^2 \leq C H^{d-2}_k\tl \boldsymbol C\sMP_e \tl_{dG}^2.
\end{align*}
By means of Proposition~\ref{prop:equivalence_dg_disc} it follows that
\begin{align*}
 |\g{w}^{(k)}|_{F\sMP[k\ell]}|_{W^{(k\ell)}}^2 \leq C H^{d-2}_k\tl \boldsymbol C\sMP_e \tl_{dG}^2 \leq C H^{d-2}_k|\p u\sMP|_{dG}^2\leq C \|\g{u}\sMP\|^2_{dG}.
\end{align*}
Using Lemma~\ref{lem:equDiscHarmonic} and Corollary~\ref{cor:equivalence_dGNorm_DHE} we can estimate 
\begin{align*}
 \|\g{u}\sMP\|^2_{dG} =\|\mathcal{H}_e\sMP(\g{w}\sMP)\|^2_{dG} \leq C \|\mathcal{H}\sMP(\g{w}\sMP)\|^2_{dG} &\leq C a_e\sMP(\mathcal{H}_e\sMP(\g{w}\sMP),\mathcal{H}_e\sMP(\g{w}\sMP)),
 \\&= C s_e\sMP(\g{w}\sMP,\g{w}\sMP)
\end{align*}
and we arrive at $|\g{w}^{(k)}|_{e}|_{W^{(k\ell)}}^2\leq C s_e\sMP(\g{w}\sMP,\g{w}\sMP)$. Since this estimate holds for the four edges of the patch, we obtain
\begin{align*}
 |\g{w}^{(k)}|_{W^{(k)}}^2 \leq C s_e\sMP(\g{w}\sMP,\g{w}\sMP) \quad \forall \g{w}\in W\sMP,
\end{align*}
where the constant $C$ is independent of $h_k$ and $H_k$, which proves the upper bound.

For the lower bound, let be $\g{w}\sMP\in W\sMP_{\Delta}$ and $ w\sMP$ its representation in the parameter domain.  We apply the lifting operator $\boldsymbol{H}_e\sMP$ to its coefficient representation $(c_i^w)_{i\in \mathcal{I}(\Gamma_e^{(k)})}$, and obtain a matrix $\boldsymbol{H}_e^{(k)}(\g w^{(k)})$ with entries $ ({c_i^H}^{(k)})_{i\in\mathcal{I}^{(k)}}$. According to (\ref{equ:coeff_dgSpline}) these entries define a spline function 
  $\g{u}^{(k)} := \{\g{u}\sMP[k,k],\{\g{u}\sMP[k,\ell]\}_{\ell\in\ifSet\sMP}\}$. We observe the estimate
  \begin{align}
  \label{equ:proof_cond_ba}
  \begin{split}
   \qh\sMP H^{d-2}_k\tl \boldsymbol{H}_e\sMP(\g w\sMP) \tl^2_{dG}& \geq C\qh\sMP H^{d-2}_k |\p u\sMP|_{dG}^2 \geq C  \|\g{u}\sMP\|^2_{dG} \geq C a_e\sMP(\g{u}\sMP,\g{u}\sMP) \\
   &\geq C a_e\sMP(\mathcal{H}_e\sMP(\g{w}\sMP),\mathcal{H}\sMP_e(\g{w}\sMP)) =C s_e\sMP(\g{w}\sMP,\g{w}\sMP),
  \end{split}
  \end{align} 
where we used inequality (\ref{equ:discrete_dg_lower}), (\ref{equ:discreteCoeff}), Lemma~\ref{lem:wellPosedDg} and the fact that $\mathcal{H}\sMP_e(\g{w}\sMP)$ minimizes the energy among given boundary data $\g{w}\sMP$. By means of Theorem~\ref{HL:thm:realMatProp}(a), we can estimate
\begin{align}
 \label{equ:proof_cond_bb}
\begin{split}
 \tl \boldsymbol{H}_e\sMP(\g w\sMP) \tl^2_{dG} &\leq C \tqh\sMP(1+\log^2\tth_k^{-1})  \sum_{\ell\in\ifSet\sMP}\biggl(\tl \g{w}^{(k)}|_{F\sMP[k\ell]} \tl_{F\sMP[k\ell]}^2 + \frac{\delta}{\tth_{k\ell}}\sum_{i=1}^{M\sMP[\ell k]}|c_i\sMP[k,\ell]-\tilde{c}_i\sMP[k,\ell]|^2 \tth_\ell \biggr)\\
 &\leq  C \tqh\sMP (1+\log^2\tth_k^{-1}) H^{2-d}_k |\g{w}\sMP|^2_{W\sMP}.
 \end{split}
\end{align}
Combining (\ref{equ:proof_cond_ba}) and (\ref{equ:proof_cond_bb}) gives 
\begin{align*}
 s_e\sMP(\g{w}\sMP,\g{w}\sMP) \leq C \qh\sMP\tqh\sMP (1+\log^2\tth_k^{-1}) |\g{w}\sMP|^2_{W\sMP}
\end{align*}

Due to (\ref{equ:h/H}), we have $\tth_k\approx h_k/H_k$, and since $H_k\approx H_\ell$ we obtain $\tqh\approx \qh$. Taking the maximum over all patches proves the upper bound. By applying Theorem~\ref{HL:thm:abstractBDDC}, the condition number bound follows.
\qed
\end{proof}

Theorem~\ref{HL:thm:CondCoeffScal} provides the theoretical basis for the numerical results obtained in \cite{HL:HoferLanger:2016a} for the two-dimensional case with only vertex primal variables. The numerical results indicate that this bound also holds for continuous edge averages as primal variables and for three-dimensional problems with additional interface or edge averages. Although the presented proof does not cover the case of jumping diffusion coefficients, we  also observed  robustness of the condition number in such cases in \cite{HL:HoferLanger:2016a}. Note that the condition number bound obtained in Theorem~\ref{HL:thm:CondCoeffScal} depends on the ratio $h_\ell/h_k$. However, numerical results do not reproduce this behaviour, cf., Section\,4.3 in \cite{HL:HoferLanger:2016a}. 
We point out that the presented analysis does not answer the dependence on the B-Spline degree $p$. The numerical experiments presented in \cite{HL:HoferLanger:2016a} indicate that the condition number depends on the degree, but very moderately in a logarithmic way in 2d. In 3d we observe a linear dependence. Similar results have been obtained in \cite{HL:VeigaChoPavarinoScacchi:2012a} for the cG BDDC-IgA preconditioner, see also \cite{HL:HoferLanger:2016b} for the cG-IETI-DP method.
Despite its not explicitly highlighted in this paper, the condition number bound proved here is not independent of the dG parameter $\delta$, in contrast to the bound given in \cite{HL:DryjaGalvisSarkis:2013a} for the FE equivalent. The dG parameter is contained in the constant appearing in Theorem~\ref{HL:thm:realMatProp}(a) and in \eqref{equ:discrete_dg_lower}. In the numerical experiments $\delta$ is chosen to be $(p+1)(p+d)$, where $d$ is the dimension. Hence, the influence of $\delta$ on the algorithm is implicitly contained in the experiments about the $p$-dependence.    
\section{Conclusion}
\label{sec:conclusion}
In this paper, we have considered non-overlapping domain decomposition methods based on the tearing and interconnecting strategy for IgA in combination with dG on the patch interfaces. We have shown that  the condition number of the preconditioned linear system obtained by the dG-IETI-DP method and the corresponding BDDC method behave quasi-optimal with respect to  $H/h:=\max_k(H_k/h_k)$. The analysis was done for the two-dimensional case having only vertex primal variables and homogeneous diffusion coefficients. The extension to 3d and other primal variables is certainly possible, but even more technical. Numerical examples presented in \cite{HL:HoferLanger:2016a} confirm the quasi optimal condition number bound obtained here. Note, the bound in Theorem~\ref{HL:thm:CondCoeffScal} depends on the ratio of the neighbouring mesh sizes $h_\ell/h_k$. However, the numerical examples in \cite{HL:HoferLanger:2016a} indicate that the condition number is also independent of $h_\ell/h_k$. 

\section*{Acknowledgements}
This work was supported by the Austrian Science Fund (FWF) under the grant W1214, project DK4. This support is gratefully acknowledged. Moreover, the author wants to thank Prof. Ulrich Langer for the valuable comments and support during the preparation of the paper.

\bibliographystyle{abbrv}

\bibliography{dgIETI_analysis.bib}

\end{document}